\documentclass[reqno,a4paper]{article}
\usepackage{a4wide}
\usepackage[centertags]{amsmath}
\usepackage[normalem]{ulem}
\usepackage[british,american]{babel}
\usepackage{amsfonts,amssymb,amsthm,amsopn,amscd}
\usepackage{eucal,dsfont,accents,bbm,bibgerm,esint,paralist,url,verbatim,wasysym,units}
\usepackage[draft]{optional}
\usepackage{cite}
\usepackage[usenames]{color}
\definecolor{citegreen}{rgb}{0,0.6,0}
\definecolor{refred}{rgb}{0.8,0,0}
\usepackage[colorlinks, citecolor=citegreen, linkcolor=refred]{hyperref}

\title{Smooth long-time existence of Harmonic Ricci Flow\\ on surfaces}
\author{Reto Buzano and Melanie Rupflin}
\date{}

\providecommand{\norm}[1]{\Vert#1\Vert} 
\providecommand{\Norm}[1]{\big\Vert#1\big\Vert} 
 
\providecommand{\abs}[1]{\vert#1\vert} 
\providecommand{\Abs}[1]{\big\vert#1\big\vert} 
\providecommand{\Aabs}[1]{\left\vert#1\right\vert} 
\providecommand{\abso}[1]{\vert#1\vert_{g_0}} 
\providecommand{\Abso}[1]{\big\vert#1\big\vert_{g_0}}

\providecommand{\Scal}[1]{\big\langle #1\big\rangle}

\DeclareMathOperator{\Vol}{vol}
\DeclareMathOperator{\inj}{inj}
\DeclareMathOperator{\tr}{tr}

\DeclareMathOperator{\im}{im}
\DeclareMathOperator{\id}{id}
\DeclareMathOperator{\Sym}{Sym}
\DeclareMathOperator{\Real}{Re}

\DeclareMathOperator{\arsinh}{arsinh}

\newcommand{\Rm}{\mathrm{Rm}}
\newcommand{\Ric}{\mathrm{Ric}}
\newcommand{\R}{\ensuremath{{\mathbb R}}}
\newcommand{\N}{\ensuremath{{\mathbb N}}}

\newcommand{\Z}{\ensuremath{{\mathbb Z}}}
\newcommand{\Hol}{\mathcal{H}}
\newcommand{\newhol}{\mathrm{H}}
\newcommand{\Conf}{\mathcal{C}}
\newcommand{\Col}{\mathrm{Col}}
\newcommand{\half}{\tfrac{1}{2}}
\newcommand{\TMint}{\int_0^T\!\!\int_M}

\newcommand{\na}{\nabla}
\newcommand{\nao}{\nabla_{g_0}}

\newcommand{\dmo}{d\mu_{g_0}}
\newcommand{\Deltao}{\Delta_{g_0}}
\newcommand{\deltao}{\delta_{g_0}}

\newcommand{\eps}{\varepsilon}
\newcommand{\tensor}{\otimes}
\newcommand{\ddt}{\frac{d}{dt}}
\newcommand{\M}{\ensuremath{{\mathcal M}}_{c}}
\newcommand{\tf}{\overset{\circ}}

\newcommand{\pt}{\partial_t}
\newcommand{\peps}{\partial_\eps}
\newcommand{\pxi}{\partial_{x_i}}
\newcommand{\pxj}{\partial_{x_j}}

\newcommand{\Om}{\Omega}

\newcommand\ubar[1]{\underaccent{\bar}{#1}}
\theoremstyle{plain}
\newtheorem{lemma}{Lemma}[section]
\newtheorem{prop}[lemma]{Proposition}
\newtheorem{thm}[lemma]{Theorem}
\newtheorem{cor}[lemma]{Corollary}

\newtheorem{remark}[lemma]{Remark}

\numberwithin{equation}{section}

\begin{document}
\maketitle
\begin{abstract}
We prove that at a finite singular time for the Harmonic Ricci Flow on a surface of positive genus both the energy density of the map component \emph{and} the curvature of the domain manifold have to blow up simultaneously. As an immediate consequence, we obtain smooth long-time existence for the Harmonic Ricci Flow with large coupling constant.
\end{abstract}

\section{Introduction and Results}
In this article we study Harmonic Ricci Flow (i.e.~Ricci Flow coupled with Harmonic Map Heat Flow), which was introduced by the first author in \cite{M09,M12}, with some special cases previously studied in \cite{List,Lott}. This coupled flow system is natural from both a geometric and an analytic point of view and is related to various other extended Ricci Flow systems, as well as to the Ricci Flow on warped products and problems arising in physics, see for example the discussion in the introduction of \cite{M09}. One of its key features is that on the one hand, in special situations, it behaves less singular than the two flows considered separately, while on the other hand a large portion of the Ricci Flow techniques carry over almost directly to this coupled system. For these reasons, this relatively new flow has gained the attention of many authors recently, studying the flow usually in general dimensions. Among some of the most interesting results are a compactness theorem for solutions of the coupled flow \cite{Wi15}, Sobolev inequalities along the flow \cite{Ab15b,FZ15}, gradient and heat kernel estimates \cite{B13,B13b,BT13,FZ15b}, eigenvalue estimates, entropies and other monotonicity formulas \cite{Ab14d,Ab15,GPT13,Li13,M10}, properties of soliton solutions \cite{GPT14b,Ta14,Ta15,Ta15b}, as well as various different Harnack inequalities \cite{CGT14,Fa13,GH14,GI14,GI14b,Wa12,Zh13}. The goal of the present paper is to focus on the special case where the domain manifold is two-dimensional where much stronger results can be obtained.\\

Let $M$ be a smooth closed oriented surface of positive genus $\gamma > 0$ and $(N,g_N)$ a smooth closed Riemannian manifold of arbitrary dimension. By Nash's embedding theorem, we can assume without loss of generality that $(N,g_N)$ is isometrically embedded in some $\R^n$. Let $g(t)$ be a family of smooth Riemannian metrics on $M$ and $\phi(t)$ a family of smooth maps from $M$ to $N$. We call $(g(t),\phi(t))$ a solution of the (volume preserving) Harmonic Ricci Flow \cite{M12}, if it satisfies
\begin{equation}\label{eq.RH}
\left\{\begin{aligned}
\pt g&=-2\Ric_g+2\alpha\; d\phi\tensor d\phi+\Big(\fint_M\tr_g(\Ric_g-\alpha\; d\phi\tensor d\phi) d\mu_g\Big)g =: T(\phi,g),\\
\pt \phi&=\tau_g(\phi),
\end{aligned}\right.
\end{equation}
where $\Ric_g$ denotes the Ricci curvature of $(M,g)$, $\tau_g(\phi)=\tr_g(\na d\phi)$ is the tension field of $\phi$ and $\alpha$ is a (possibly time-dependent) positive coupling constant. We will always assume that $\alpha(t)\in[\ubar{\alpha},\bar{\alpha}]$, where $0<\ubar{\alpha}\leq \bar{\alpha}<\infty$.\\

Typically, geometric flows can develop finite time singularities. In the case of the Harmonic Ricci Flow, this could be caused either by energy concentration of the map component $\phi(t)$ or by a degeneration of the evolving domain metric $g(t)$ -- or by both phenomena happening at the same time, as described more precisely in \eqref{eq.degenerations} below. However, the goal of this article is to prove that in the case where the domain manifold is two-dimensional and the coupling constant $\alpha(t)$ is sufficiently large, such finite time singularities cannot appear at all. Our main result is the following.

\begin{thm}\label{thm:large-alpha}
Let $\alpha(t)\in[\ubar{\alpha},\bar{\alpha}]$ be a smooth coupling function satisfying
\begin{equation}\label{eq.alphabound}
\ubar{\alpha}> 2\max \big\{ K(\tau)\mid \tau\subset T_pN \textrm{ two plane, } p\in N \big\},
\end{equation}
where $K(\tau)$ denotes the sectional curvature of the target manifold $N$ in direction $\tau$. Then every solution $(g,\phi)$ of \eqref{eq.RH} is defined and smooth for all times $t\geq 0$. 
\end{thm}

In \cite[Corollary 5.3]{M12}, the first author of this article showed that if the coupling constant $\alpha(t)$ is smooth and bounded away from zero, then for a solution $(g,\phi)$ of the Harmonic Ricci Flow (in arbitrary dimension), energy concentration of $\phi(t)$ cannot happen as long as the curvature of $g(t)$ stays bounded. Here, we show that in the case of a two-dimensional domain, the converse holds as well. That is, it is not possible that the (Gauss) curvature $K_{g(t)}$ of $g(t)$ blows up while the energy density of $\phi(t)$ remains bounded. In particular, we have the following theorem.

\begin{thm}\label{thm:blow-up}
Let $(g,\phi)$ be a solution of \eqref{eq.RH} defined and smooth on a maximal time interval $[0,T)$ and with a smooth coupling function $\alpha$ that is bounded away from zero. If $T<\infty$, then both the map \emph{and} the metric component must blow up
\begin{equation}\label{eq.degenerations}
\limsup_{t\nearrow T}\max_{x\in M} \Abs{K_{g(t)}(x)}=\infty \quad\text{ and }\quad \limsup_{t\nearrow T}\max_{x\in M} \, \frac12 \Abs{d\phi(x,t)}_{g(t)}^2=\infty.
\end{equation}
\end{thm}
Theorem \ref{thm:large-alpha} then follows directly from Theorem \ref{thm:blow-up}, as the assumption \eqref{eq.alphabound} prevents $\abs{d\phi(x,t)}_{g(t)}^2$ from blowing up. In fact, by the Bochner-formula, $\abs{d\phi(x,t)}_{g(t)}^2$ is uniformly bounded (in space and time) in terms of its initial value and $\bar{\alpha}$, $\ubar{\alpha}$ satisfying \eqref{eq.alphabound}. For the unnormalised flow, this was derived in \cite[Proposition 5.6]{M12}. It is easy to check that a similar argument holds for the normalised flow as well; we will carry this out in the very last section of this article for completeness.\\

Observe that since $M$ is a surface, we have $\Ric_g=\half R_g\cdot g=K_g\cdot g$, where $R_g$ is the scalar and $K_g$ the Gauss curvature of $g$. Consequently, the mean value $\fint_M\tr_g(\Ric) d\mu_g=2(\Vol(M,g))^{-1}\cdot \int_M K_g d\mu_g$ depends only on the genus of the surface and on the volume of $(M,g)$, which is constant along the flow \eqref{eq.RH}. Thus, we will always assume that the volume of the initial surface is equal to the one of surfaces of constant Gauss curvature
\begin{equation}\label{eq.Rbar}
\bar{K} = \begin{cases}
-1, & \text{ if }\gamma\geq 2,\\
0, & \text{ if }\gamma=1,
\end{cases}
\end{equation}
while assuming unit volume in the torus case. That is, we assume that the volume of the initial surface $(M,g(0))$, and hence any $(M,g(t))$ is given by
\begin{equation}\label{eq.initialvolume}
\Vol(M,g) = \begin{cases}
4\pi(\gamma-1), & \text{ if }\gamma\geq 2,\\
1, & \text{ if }\gamma=1.
\end{cases}
\end{equation}
This assumption implies that $\fint_M\tr_g(\Ric) d\mu_g=\bar{R}=2\bar{K}$. Proving our results only for these initial manifolds constitutes no loss of generality, as the result for initial data of general volume then simply follows by rescaling. More precisely, by rescaling both the metric and the time variable of the map component, we obtain a flow similar to \eqref{eq.RH} but with an additional constant factor in front of the tension field in the second equation. This factor however does not non-trivially influence the arguments that follow and we will therefore not write it for the sake of readability of the arguments.\\

Let us now give an overview of the proof of Theorem \ref{thm:blow-up} and an outline of this paper. The main idea behind our analysis of the flow \eqref{eq.RH} is that, as explained in Section \ref{sec.split}, one can always split a flow of metrics on a surface into a \emph{conformal} part, the \emph{pull-back by diffeomorphisms} and a \textit{horizontal} movement. After pulling back with a suitable family of diffeomorphisms, we can thus obtain a new flow (again denoted by $(g,\phi)$), which is constructed in such a way that $g(t)$ is of the form 
\begin{equation}\label{1.gg0}
g(t)=e^{2u(t)}g_0(t),
\end{equation}
where $g_0(t)$ moves only in horizontal direction, see Section \ref{sec.split} for details. This new flow solves the evolution equations 
\begin{equation}\label{eq.flow}
\left\{\begin{aligned}
\pt g&=T(\phi,g)-L_Xg,\\
\pt \phi&=\tau_g(\phi)-d\phi\cdot X,
\end{aligned}\right.
\end{equation}
where $X$ is the generating vector field of the diffeomorphisms.\\

As we shall see, the evolution of the horizontal curve $g_0$ is determined mainly by the trace-free part $\tf{T}(\phi,g)$ of the tensor 
\begin{equation}\label{eq.Trewritten}
T(\phi,g):=2\big(\bar{K}-K_g-\alpha \bar E(t)\big)\cdot g + 2\alpha\; d\phi\tensor d\phi,
\end{equation}
where we note that  $\tf T(\phi,g)=2\alpha(d\phi\tensor d\phi -e(\phi,g)g)$ is 
a priori bounded in $L^\infty$ under the assumption that the energy density is uniformly bounded, and $g_0$ can be controlled by methods developed by P. Topping and the second author in \cite{RT, R-existence, RT-neg, RT-horizontal}. Here and in the following, we use the standard notation for the energy density and total Dirichlet energy of $\phi$, namely $e(\phi,g):=\half\abs{d\phi}_g^2$ and $E(\phi,g):=\half\int_M \abs{d\phi}_g^2\; d\mu_g$, as well as the mean energy density $\bar{E}(\phi,g)=\half\fint_M \abs{d\phi}_g^2 \; d\mu_g = E(\phi,g)/\Vol(M,g)$. Conversely, the trace part of $T(\phi,g)$ essentially causes only a conformal movement, which we will control by more classical methods, inspired by Struwe's analysis \cite{Struwe} of two-dimensional Ricci Flow.\\

Since a solution of \eqref{eq.RH} and the corresponding solution of \eqref{eq.flow} differ only by the pull-back by diffeomorphisms, geometric quantities such as the energy density or the curvature agree and we can derive Theorem \ref{thm:blow-up} from the following key result.

\begin{prop}\label{prop:main-result}
Let $(g,\phi)$ be a solution of \eqref{eq.flow} (where the choice of $X$ is explained in Section \ref{sec.split}) and assume that on an interval $[0,T)$, $T<\infty$, the energy density of $\phi$ is uniformly bounded, that is
\begin{equation} \label{ass:energy-density} 
\sup_{x\in M, t\in [0,T)} \frac12 \Abs{d\phi(x,t)}_{g(t)}^2<\infty.
\end{equation}
Then both the curvature and the injectivity radius of $g(t)$ are uniformly bounded, 
\begin{equation*}
\sup_{x\in M, t\in [0,T)} \Abs{K_{g(t)}(x)}<\infty, \qquad\text{ and }\qquad \inf_{t\in [0,T)}\inj(M,g(t)) >0,
\end{equation*}
and thus the solution $(g,\phi)$ of \eqref{eq.flow} can be extended smoothly past time $T$.
\end{prop}

The proof of Proposition \ref{prop:main-result} starts in Section \ref{sec.horizontal}, where we show that the evolution of the underlying conformal structure, described by the horizontal curve $g_0(t)$, is well controlled and in particular that the injectivity radius of $g_0(t)$ is a priori bounded away from zero on any given time interval of finite length. This corresponds to saying that $g_0(t)$ does not degenerate in Teichm\"uller space and is trivially true for surfaces of genus $\gamma=1$ as their Teichm\"uller space is complete with respect to the Weil-Petersson metric and the $L^2$-length of $t \mapsto g_0(t)$ is bounded, compare \eqref{est:uniform-Linfty-T}. In the case of genus at least two, more work is needed. Intuitively, the evolution of $g_0(t)$ is driven by the horizontal component of the trace-free part of $T(\phi,g)$, which is
\begin{equation} 
\tf T(\phi,g)=2\alpha\big(d\phi\tensor d\phi-e(\phi,g)g\big)=\alpha\Real(\Phi(\phi,\mathbbm{c})),
\end{equation}
where $\Phi(\phi,\mathbbm{c})=(\vert \phi_x\vert^2-\vert\phi_y\vert^2-2i\langle \phi_x,\phi_y\rangle) dz^2$, $z\in \mathbbm{c}$, is the Hopf-differential of $\phi$ on the Riemann surface $(M,\mathbbm{c})$, and $\mathbbm{c}$ is a complex structure compatible with $g$ (and thus $g_0$). We can thus apply techniques developed for the study of Teichm\"uller Harmonic Map Flow by  Peter Topping and the second author, see \cite{RT,R-existence, RT-neg,RT-horizontal, RT-global}.\\

In Section \ref{sec.struwe}, based on the estimates obtained for $g_0(t)$, we then analyse the evolution of the conformal factor $u(t)$ in \eqref{1.gg0} following the approach of Struwe \cite{Struwe}. Let us remark that while the Dirichlet energy of the map component is decreasing according to 
\begin{equation}\label{energy-decay}
\begin{aligned}
\ddt E(\phi,g)&=-\int_M\pt \phi\cdot \tau_g(\phi)d\mu_g-\frac{1}{4\alpha}\int_M\Scal{ \tf T(\phi,g),\pt g}_{g}d\mu_g\\
&=-\Norm{\tau_g(\phi)}_{L^2(M,g)}^2-\frac{1}{4\alpha}\Norm{\tf T(\phi,g)}^2_{L^2(M,g)},
\end{aligned}
\end{equation}
the Liouville energy which controls the conformal factor $u(t)$ is not monotonically decreasing (unlike in the classical Ricci Flow case). Nevertheless, we can still derive bounds on the Liouville energy and thus on the $H^1$-norm of $u$ and then, in a second step, also derive $H^2$-bounds on $u$. All of this is carried out with respect to the evolving metric $g_0(t)$.\\

Finally, in Section \ref{sec.curv}, higher regularity estimates for $u$ are derived via bootstrapping arguments. In view of the Gauss equation
\begin{equation}\label{Gausseq}
K_g = e^{-2u}\big(\bar{K}-\Deltao u\big),
\end{equation}
these estimates will then yield the desired curvature estimates and hence finish the proof of Proposition \ref{prop:main-result} (and thus also of the Theorems \ref{thm:large-alpha} and \ref{thm:blow-up}).\\

Let us remark that there are initial conditions for which the flow cannot both exist smoothly for all time and converge smoothly as $t\to\infty$. Hence, at time infinity, some singularities must occur for these initial data. As the energy density of the map component stays uniformly bounded, bubbling in the map component cannot happen and hence the domain must develop these singularities. The precise asymptotics will be studied elsewhere.

\begin{remark}
We exclude the case where the domain manifold is a two-sphere from this article as it requires different methods. While the Teichm\"{u}ller space of a two-sphere consists only of one point and thus there are no horizontal movements to be considered, an additional difficulty arises from the non-compactness of the M\"{o}bius group of oriented conformal diffeomorphisms, which in particular prevents the argument in Section \ref{sec.struwe} from being carried over to this case directly. Moreover, as the Liouville energy is not decreasing, it seems that also Struwe's gauge fixing trick \cite{Struwe} cannot be adapted directly to this situation.
\end{remark}

\textbf{Acknowledgements.} Parts of this work were carried out during a visit of RB at the University of Leipzig as well as a visit of MR at Queen Mary University of London. We would like to thank the two universities for their hospitality. Furthermore, we would like to thank the referees for their comments which helped us improve this article. RB also acknowledges financial support from the EPSRC grant number EP/M011224/1.

\section{Splitting the flow into its conformal, horizontal and Lie-derivative parts}\label{sec.split}
In this section, we will explain how the two-dimensional Harmonic Ricci Flow -- or in fact any flow of Riemannian metrics on a closed surface -- can be split into evolutions in conformal, horizontal and Lie-derivative directions.\\

In the following, we let $\Sym_{+}^2(T^*M)$ denote the space of smooth Riemannian metrics on $M$ and let $\M$ be the subset of metrics with constant Gauss curvature $\bar{K}=c\in\{-1,0\}$, as in \eqref{eq.Rbar}, which also satisfy the volume constraint \eqref{eq.initialvolume}.

\subsection{Splitting general flows on surfaces of positive genus}
Let $M$ be a closed oriented surface of genus $\gamma>0$. 
Following the approach of Tromba \cite{Tromba}, we first recall how 
 the space of symmetric two-tensors $\Sym^2(T^*M)$ (which is the tangent space to the space of all Riemannian metrics on $M$)
splits into subspaces that correspond to the evolutions by conformal changes, pull-backs by diffeomorphisms or horizontal curves, respectively, we refer to \cite{Tromba} for details.\\

By the uniformisation theorem, every metric $g$ on $M$ is conformal to a metric $\bar{g}\in\M$. Note that under the volume constraint \eqref{eq.initialvolume}, the constant Gauss curvature metric $\bar{g}$ is unique.  We split $\Sym^2(T^*M)$ into the following direct sum:
\begin{equation}\label{2.prelimsplit}
\Sym^2(T^*M) = \Conf(\bar{g}) \oplus T_{\bar{g}}\M,
\end{equation}
where $\Conf(\bar{g}):=\{\varrho\cdot \bar{g} \mid \varrho\in C^\infty(M,\R)\}$. In order to split $T_{\bar{g}}\M$ further, denote by $\delta_{\bar{g}}$ the divergence acting on $\Sym^2(T^*M)$ with adjoint $\delta_{\bar{g}}^*$. Recall that since $\delta_{\bar{g}}^*$ is overdetermined elliptic, we have the $L^2$-orthogonal splitting
\begin{equation}\label{eq.barsplit}
T_{\bar{g}}\M=\im \delta_{\bar{g}}^* \oplus \big(\ker \delta_{\bar{g}} \cap T_{\bar{g}}\M\big).
\end{equation}
As $\delta_{\bar{g}}^*X = - L_X\bar{g}$, we obtain that
\begin{equation*}
\im \delta_{\bar{g}}^* = \big\{L_X\bar{g} \mid X\in\Gamma(TM)\big\}
\end{equation*}
is the space of Lie derivatives of $\bar{g}$. Moreover, it is well known that the horizontal space
\begin{equation*}
\big(\ker \delta_{\bar{g}} \cap T_{\bar{g}}\M\big) =: \Hol(\bar{g})=\big\{h\in\Sym^2(T^*M) \mid \delta_{\bar{g}}(h)=0=\tr_{\bar{g}}(h)\big\}
\end{equation*}
consists of the real parts of holomorphic quadratic differentials. Recall also that $\Hol(\bar{g})$ is $L^2$-orthogonal to $\Conf(\bar{g})$, which means that overall we have a splitting
\begin{equation}\label{2.mainsplit}
\Sym^2(T^*M) = \big(\Conf(\bar{g}) \oplus \im \delta_{\bar{g}}^*\big) \oplus \Hol(\bar{g}),
\end{equation}
where the last factor splits off $L^2$-orthogonally. This last property allows us to obtain the horizontal component (and this component only) by an orthogonal projection, which we denote by $P_g^\Hol:\Sym^2(T^*M)\to \Hol(g)$.
\begin{remark}\label{rem.normalX}
In the case of genus at least two, $X \mapsto L_X\bar{g}$ is injective, as there are no Killing fields. Hence the above splitting \eqref{2.mainsplit} uniquely determines $X$. For genus one, there is a two-dimensional space of Killing fields, but we can determine $X$ uniquely if we require the normalisation
\begin{equation}\label{2.normalisationX}
m(X) := \int_M P_{x,y}(X(x)) \; d\mu_{\bar{g}}(x) = 0,
\end{equation}
where $P_{x,y}(X(x))$ denotes the parallel transport of $X(x)\in T_xM$ to the \emph{fixed} tangent space $T_yM$ for any choice of $y\in M$. Viewing a flat torus as parallelogramm $P\subset \R^2$ with opposite sides identified this is equivalent to requiring that the integral over $P$ of the corresponding components of the vectorfield is zero. 
\end{remark}

This splitting of $\Sym^2(T^*M)$ translates to a corresponding splitting of a geometric flow as follows. Let $(\tilde g(t))_{t\in[0,T)}$ be any smooth curve of metrics satisfying the volume constraint \eqref{eq.initialvolume}. Then, by the uniformisation theorem, there exists a unique curve of metrics $(\bar g(t))_{t\in[0,T)}\subset \M$ and a unique function $v\in C^\infty(M\times[0,T))$ such that 
\begin{equation}\label{eq.Deftildeg}
\tilde{g}(t)=e^{2v(t)}\bar{g}(t).
\end{equation}
The splitting of the tangent space $T_{\bar g}\M=\im \delta_{\bar{g}}^*\oplus \Hol(\bar g)$ described above then yields the following result.

\begin{lemma}\label{2.exg0lemma}
Let $\bar g$ be a smooth curve in $\M$. Then there exists a unique smooth family of diffeomorphisms $f_t:M\to M$, generated by a vectorfield satisfying \eqref{2.normalisationX} in case $\gamma=1$, with $f_0=\id$ and a unique \emph{horizontal} curve $(g_0(t))_{t\in[0,T)}$, i.e. a curve of metrics satisfying
\begin{equation*}
\pt g_0(t)\in \Hol(g_0(t)),\quad\text{ for every }t\in[0,T),
\end{equation*}
such that $\bar{g}(t)=f_t^* g_0(t)$.
\end{lemma}

\begin{proof}
Let $\bar{g}(t)$ be a smooth curve in $\M$. Then by \eqref{eq.barsplit}, there exists a unique vectorfield $\bar X(t)$, in case of a torus additionally normalised by \eqref{2.normalisationX}, and a unique tensor $h(t)\in \Hol(\bar g(t))$ so that 
\begin{equation*}
\pt \bar g=h+L_{\bar X}\bar g.
\end{equation*}
On the other hand, given any family of diffeomorphisms $f_t$ generated by some vectorfield $X$ and a curve of metrics $g_0$ with $\bar g=f_t^*g_0$ we  write
\begin{equation*}
\pt \bar g=f_t^*\big(\pt g_0 + L_Xg_0\big)=f_t^*(\pt g_0)+L_{f_t^*X}\bar g
\end{equation*}
and recall that $f_t^*\Hol(g_0)=\Hol(\bar g)$. Thus $g_0$ is horizontal if and only if 
\begin{equation*}
f_t^*(\pt g_0(t))=h(t)=P_{\bar g(t)}^{\Hol}(\pt \bar g(t)), \quad \text{ for every } t,
\end{equation*}
and consequently $
L_{f_t^*X}\bar g=L_{\bar X}\bar g$. Note that a vectorfield $X$ satisfies \eqref{2.normalisationX} for $g_0$ precisely if $f_t^* X$ satisfies \eqref{2.normalisationX} for $\bar g$, so we must indeed have that
\begin{equation*}
f_t^*X=\bar X.
\end{equation*}
This allows us in a first step to determine the desired smooth family of diffeomorphisms $f_t$, e.g. by using that their inverses are generated by $-f_t^*X=-\bar X$. In a second step, we then obtain the smooth horizontal curve $g_0$ by integrating
\begin{equation*}
\pt g_0=(f_t^{-1})^*P_{\bar g}^{\Hol}(\pt \bar g)=(f_t^{-1})^*\big(\pt \bar g-L_{\bar X}\bar g\big).\qedhere
\end{equation*}
\end{proof}

Now, given $f_t$ as in the above lemma, notice that $\tilde{g}(t)$ as in \eqref{eq.Deftildeg} can be rewritten as
\begin{equation}
\tilde g(t)=e^{2v}f_t^*g_0=f_t^*(e^{2u}g_0),  \quad u=v\circ f_t^{-1}.
\end{equation}
Hence, instead of directly considering a solution $(\tilde{g}(t),\tilde{\phi}(t))_{t\in[0,T)}$ of the original flow equation \eqref{eq.RH}, we can instead analyse the pulled-back flow
\begin{equation*}
g(t)=(f_t^{-1})^*\tilde g(t)=e^{2u}g_0(t), \qquad \phi(t) =\tilde{\phi}(t) \circ f_t^{-1},
\end{equation*}
which has the advantage that the metric component of this flow only moves by conformal changes and horizontal movements. As mentioned in the introduction, if $(\tilde{g},\tilde{\phi})$ satisfies the flow equation \eqref{eq.RH}, then the pulled-back flow $(g,\phi)$ satisfies the evolution equation \eqref{eq.flow}, where $X$ is the vector field generating the diffeomorphisms $f_t$ of Lemma \ref{2.exg0lemma}.

\subsection{Equations for the components of Harmonic Ricci Flow}
From the above discussion, and in particular \eqref{eq.flow}, we obtain the evolution equation
\begin{align*} 
\pt g_0&=e^{-2u}\big(\pt g-2\pt u\cdot g\big)=e^{-2u}\big(T(\phi,g)-L_X(e^{2u}g_0)-2\pt u\cdot g\big)\\
&=e^{-2u}\,\tf T(\phi,g)-L_Xg_0+\big({-2\pt u}+\half\tr_g(T(\phi,g))-2du(X)\big)\cdot g_0,
\end{align*}
or, in the order of \eqref{2.mainsplit}, 
\begin{equation}\label{2.splittingtfT}
e^{-2u}\,\tf T(\phi,g) = (\varrho g_0 + L_Xg_0) + \pt g_0,
\end{equation}
where 
\begin{equation}\label{eq.rho}
\varrho = 2\pt u-\half\tr_g(T(\phi,g))+2du(X).
\end{equation}
As we know from the above that $\Hol(g_0)$ stands $L^2(M,g_0)$-orthogonal to both the space of all Lie-derivatives and to all conformal directions, we can obtain the velocity $\pt g_0$ of the horizontal curve $g_0$ as the $L^2(M,g_0)$-orthogonal projection $P_{g_0}^\Hol$ onto $\Hol(g_0)$ of the tracefree part of $e^{-2u}T(\phi,g)$,
\begin{equation}\label{2.dtg0}
\pt g_0=P_{g_0}^\Hol\big(e^{-2u}\,\tf T(\phi,g_0)\big)=\Real \big(P_{g_0}^\newhol(\alpha e^{-2u}\Phi(\phi,g_0))\big)
\end{equation}
or equivalently as the real part of the projection $P_{g_0}^\newhol$ of the (weighted) Hopf-differential $\Phi(\phi,g_0)$ to the space of holomorphic quadratic differentials $\newhol(M,g_0)$.\\

Next, we characterise $\varrho$ and $X$ in \eqref{2.splittingtfT}. From \eqref{2.prelimsplit}, we know that for $t\in[0,T)$, we have\begin{equation*}
e^{-2u(t)}\,\tf T(\phi,g)(t)-\varrho(t)g_0(t)\in T_{g_0(t)}\M
\end{equation*}
and hence, we know that $(DR)(g_0)(e^{-2u}\,\tf T(\phi,g)-\varrho(t)g_0)=0$. Since for any metric $g_0\in \M$ 
\begin{equation*}
DR(g_0)(k)=-\Delta_{g_0}(\tr_{g_0}(k))+\delta_{g_0}\delta_{g_0}k-\bar{K}\tr_{g_0}(k),
\end{equation*}
where $\bar{K}\in\{-1,0\}$ as in \eqref{eq.Rbar}, we can characterise $\varrho(t)$ equivalently as the unique solution of the elliptic PDE
\begin{equation} \label{eq:rho}
-\Delta_{g_0}\varrho-2\bar{K}\cdot \varrho=\delta_{g_0}\delta_{g_0}\big(e^{-2u}\,\tf T(\phi,g)\big) \quad \text{with} \quad \int_M \varrho\, \dmo = 0.
\end{equation}
We remark that the condition $\int_M \varrho\, \dmo = 0$ is automatically satisfied in the case of genus $\gamma \geq 2$, but is required in the torus case due to our volume constraint. We will see that this characterisation of $\varrho$ not only implies that $\varrho$ is smooth on $M \times [0,T)$, provided that both $g$ and $\phi$ are smooth, but it will also allow us to derive estimates on Sobolev norms of $\varrho$ in terms of the appropriate Sobolev norms of the right hand side. Furthermore, thanks to the control on the underlying constant curvature metric $g_0$ derived in the next section, such estimates will be valid with uniform constants.\\

Next, using \eqref{2.splittingtfT} and \eqref{2.dtg0}, and recalling that tensors in $\Hol(g_0)$ are divergence-free, we know that the vector field $X\in\Gamma(TM)$ satisfies $e^{-2u }\tf T(\phi,g)-\varrho g_0-L_Xg_0\in \Hol(g_0)$ and thus\begin{equation} \label{eq:X}
\delta_{g_0}\delta_{g_0}^*X=-\delta_{g_0}L_Xg_0 =-\delta_{g_0}\big(e^{-2u}\,\tf T(\phi,g)-\varrho g_0\big).
\end{equation}
If $(M,g_0)$ has negative curvature, there are no non-trivial Killing-fields and thus $X=X(\phi,g)$ is uniquely determined by the above equation. On a torus, this is still true up to normalisation of $X$ as explained in Remark \ref{rem.normalX} above.\\

We point out that so far, we have never used the precise definition of $T(\phi,g)$, but only assumed that it is a diffeomorphism invariant tensor, which is smooth if $g$ and $\phi$ are smooth. All in all, we have thus proved the following proposition.

\begin{prop}
Let $T(\phi,g)$ be so that $T(f^*(\phi,g))=f^*(T(\phi,g))$ for every diffeomorphism $f:M\to M$, $(g(t)=e^{2u(t)}g_0(t), \phi(t))$ be a solution of \eqref{eq.flow} on a surface of genus $\gamma\geq 1$ where $g_0$ is a horizontal curve. Then, the following evolution equations hold: 
\begin{align*}
\pt g_0&=P_{g_0}^\Hol\big(e^{-2u}\,\tf T(\phi,g_0)\big),\\
\pt u&=\tfrac14 \tr_{g}(T(\phi,g))+\half\varrho-du(X),\\
\pt \phi&=\tau_{g}(\phi)-d\phi(X),
\end{align*}
where $\varrho$ and $X$ are the unique solutions of \eqref{eq:rho} and \eqref{eq:X}, respectively.
\end{prop}

In the case of Harmonic Ricci Flow, where $T(\phi,g)$ is given by \eqref{eq.Trewritten}, and hence, as mentioned before,  $\tf T(\phi,g)=2\alpha(d\phi\tensor d\phi-e(\phi,g_0)g_0)$ as well as
\begin{align*}
\tfrac14 \tr_{g}(T(\phi,g)) &= \bar{K}-K_g+\alpha\big(e( \phi,g)-\bar E(t)\big)\\
&= e^{-2u}\Delta_{g_0}u+\bar{K}\big(1-e^{-2u}\big)+\alpha \big(e(\phi,g_0)e^{-2u}-\bar E(t)\big),
\end{align*}
where we used the Gauss equation \eqref{Gausseq}, we find the following evolution equations.

\begin{cor}
Let $(g(t)=e^{2u(t)}g_0(t), \phi(t))$ be a solution of \eqref{eq.flow} on a surface of genus $\gamma\geq 1$, where $T(\phi,g)$ is given by \eqref{eq.Trewritten}. Then, the following evolution equations hold: 
\begin{subequations}  \label{eq:flow-components}
\begin{align}
\pt g_0 &=P_{g_0}^\Hol\big(2\alpha \; e^{-2u}(d\phi\tensor d\phi-e(\phi,g_0)g_0)\big), \label{eq:g0}\\
\pt u  &=e^{-2u}\Delta_{g_0}u+\bar{K}\big(1-e^{-2u}\big)+\alpha \big(e(\phi,g_0)e^{-2u}-\bar E(t)\big)+\half \varrho-du(X), \label{eq:u}\\
\pt \phi &=e^{-2u}\big(\Delta_{g_0}\phi+A_{g_0}(\phi)(d\phi,d\phi)\big)-d\phi(X) \label{eq:phi}
\end{align}
\end{subequations}
where $\varrho$ and $X$ are the unique solutions of \eqref{eq:rho} and \eqref{eq:X}, respectively, and $A_{g_0}(\phi)$ denotes the second fundamental form of $(N,g_N) \hookrightarrow \R^n$ at $\phi$, so that $\tau_{g_0}(\phi) = \Delta_{g_0}\phi+A_{g_0}(\phi)(d\phi,d\phi)$.
\end{cor}

\subsection{Estimates for $\varrho$ and $X$}
From elliptic PDE theory, we obtain the following a priori estimates for $\varrho$ and $X$.

\begin{lemma}\label{lemma:rho-X}
Let $(g,\phi)$ be as in Proposition \ref{prop:main-result}. Then 
the solution $\varrho$ of \eqref{eq:rho} satisfies
\begin{equation} \label{est:rho-L2}
\sup_{t\in[0,T)}\int_M \Abs{\varrho}^p\dmo<\infty, \quad \text{ for every } 1< p<\infty,
\end{equation}
as well as 
\begin{equation} \label{est:rho-H1}
\int_M\Abs{d\varrho}_{g_0}^2\dmo\leq C\int_M\Abso{du}^2 \dmo+C\int_M\Abs{\tau_g(\phi)}^2d\mu_g+C.
\end{equation}
Similarly, the solution $X$ of \eqref{eq:X} satisfies 
\begin{equation}
\sup_{t\in[0,T)} \int_M \Big( \Abso{\nao X}^p+\Abso{X}^p\Big) \dmo<\infty, \quad \text{ for every } 1< p<\infty,
\end{equation}
and thus in particular $\sup_{t\in[0,T]} \norm{X}_{L^\infty}<\infty$ as well as 
\begin{equation}
\int_M \Abso{\nao^2 X} \dmo \leq C\int_M\Abso{du}^2 \dmo+C\int_M\Abs{\tau_g(\phi)}^2d\mu_g+C.
\end{equation}
Here and in the following, unless mentioned otherwise, $C$ always denotes a generic constant possibly changing from line to line, which is in particular allowed to depend on the coupling function $\alpha$, the time $T$, the initial data $(M,g(0))$ and $\phi(0)$ (and hence in particular the genus $\gamma$ of $M$), as well as the energy density bound in \eqref{ass:energy-density}.
\end{lemma}

\begin{proof}
Let us first remark that we have good control on the operators, thanks to the uniform estimates for $g_0$ (and in particular the uniform bound on $\inj(M,g_0)$) that we will derive in the next section (see in particular Remark \ref{rem2CorB} and Lemma \ref{lemma:inj-rad-g0}). Hence, the elliptic estimates used below all apply with constants $C$ which are independent of $t$.\\

Next, observe that due to the assumption on the energy density, the right hand sides of \eqref{eq:rho} and \eqref{eq:X} are bounded uniformly in time in the spaces $W^{-2,p}(M)$ and $W^{-1,p}(M)$, respectively. The first claim is thus essentially just $L^p$ theory applied to uniformly elliptic operators with trivial kernel.\\
 
For the $H^2$-estimate for $X$, respectively the $H^1$-estimate for $\varrho$, we observe that the divergence of the Hopf-differential is given as $\deltao(\tf T(\phi,g_0))=2\alpha\Deltao\phi\cdot d\phi=2\alpha\tau_{g_0}(\phi)\cdot d\phi$, so that the uniform bound on
\begin{equation*}
\Abs{d\phi}_g^2 = e^{-2u}\Abso{d\phi} \leq C < \infty
\end{equation*}
implies
\begin{equation}\label{est:divT}
\begin{aligned}
\Abso{\deltao(e^{-2u}\,\tf T(\phi,g_0))}^2 &\leq C\Abso{du}^2+C\Abs{\tau_{g_0}(\phi)}^2\cdot\Abso{d\phi}^2e^{-4u}\\
&\leq  C\Abso{du}^2+Ce^{2u}\Abs{\tau_{g}(\phi)}^2
\end{aligned}
\end{equation}
and thus
\begin{equation*}
\Norm{\deltao(e^{-2u}\,\tf T(\phi,g_0))}_{L^2(M,g_0)}^2\leq C\int_M\Abso{du}^2\dmo +C\Norm{\tau_g(\phi)}_{L^2(M,g)}^2.
\end{equation*}
The corresponding $L^2$-bound on the right hand side of \eqref{eq:X} implies the $H^2$-estimate for $X$, while the resulting $W^{-1,2}$-bound on the right hand side of
\eqref{eq:rho} leads to the claimed $H^1$-bound on $\varrho$.
\end{proof}

\section{Evolution in horizontal direction}\label{sec.horizontal}
As we shall see, for solutions of \eqref{eq.flow} satisfying \eqref{ass:energy-density} the evolution in horizontal direction can be very well controlled using the theory of general horizontal curves developed in the joint work \cite{RT, R-existence, RT-neg, RT-horizontal} of P. Topping and the second author.  As we shall prove later on in this section, in this setting the injectivity radius can be bounded away from zero uniformly on time intervals of bounded length. This allows us to control the metric component using results obtained in \cite{RT-horizontal} and \cite{RT-neg}, see also \cite{R-existence} for related estimates. We begin by recalling the relevant results from \cite{RT-horizontal} and \cite{RT-neg}, or rather their corollaries for the simpler setting of horizontal curves with controlled injectivity radius, to make the presentation self-contained.\\

To begin with, we recall that the fact that for holomorphic functions all $C^k$-norms are equivalent to the $L^1$-norm has an analogue for the spaces of holomorphic quadratic differentials $\newhol(M,g_0)$, $g_0\in\M$. To be more precise, in the case of bounded injectivity radius, \cite[Lemma 2.6]{RT-global} implies in particular the following.

\begin{lemma}\label{lemmaCkL1}
For every $\eps_0>0$ and $k\in \N_0$ there exists a constant $C<\infty$ such that for every $g_0\in\M$ with $\inj(M,g_0)\geq \eps_0$
\begin{equation*}
\Norm{\Real(\Omega)}_{C^k(M,g_0)} \leq C\Norm{\Omega}_{L^1(M,g_0)}, \qquad \text{ for all }\, \Omega\in\newhol(M,g_0).
\end{equation*}
\end{lemma}

Here and in the following the $C^k$-norm of a tensor $h\in \Sym^2(T^*M)$, 
\begin{equation*}
\Norm{h}_{C^k(M,g_0)}=\sum_{\ell=0}^k \sup_M \Abs{\nabla_{g_0}^\ell h}_{g_0},
\end{equation*}
is computed in terms of the Levi-Civit\`a connection on $(M,g_0)$ and its extensions to the corresponding bundles.\\

A useful consequence of this result is the following corollary.
\begin{cor}\label{corA}
Under the same assumptions as in Lemma \ref{lemmaCkL1}, we have for all quadratic differentials $\Psi$ on $(M,g_0)$
\begin{equation*}
\Norm{P_{g_0}^{\Hol}( \Real(\Psi))}_{C^k(M,g_0)} \leq C\Norm{\Psi}_{L^1(M,g_0)}.
\end{equation*}
\end{cor}
This corollary follows from Lemma \ref{lemmaCkL1} either by a direct argument, which is given below, or (for $\gamma\geq 2$) by combining the lemma with the following more general fact.
\begin{remark}
It is proven in \cite[Proposition 4.10]{RT-neg} that there exists a constant $C$ depending only on the genus $\gamma\geq 2$ of $M$ so that
$\norm{P_{g_0}^{\Hol}(\Psi)}_{L^1} \leq C \norm{\Psi}_{L^1}$ holds true for all $L^2$ quadratic differentials on $M$ and every $g_0\in \mathcal{M}_{-1}$.
\end{remark}

\begin{proof}[Proof of Corollary \ref{corA}]
Let $g_0$ be as in Lemma \ref{lemmaCkL1} and select an orthonormal basis $(\Om_j)$ of the finite dimensional space $\newhol(M,g_0)$ of holomorphic quadratic differentials. Then
\begin{equation*}
P_{g_0}^{\Hol}( \Real(\Psi))=\Real(P_{g_0}^\newhol(\Psi))=\Real\Big(\sum_j \Scal{\Psi,\Om_j}_{L^2(M,g_0)} \Om_j\Big)
\end{equation*}
so, by Lemma \ref{lemmaCkL1} 
\begin{align*}
\Norm{P_{g_0}^{\Hol}( \Real(\Psi))}_{C^k}
&\leq  \Norm{\Psi}_{L^1}\sum_j\Norm{\Om_j}_{L^\infty}\cdot \Big(\Norm{\Real(\Om_j)}_{C^k}+\Norm{\Real(i\, \Om_j)}_{C^k}\Big)
\\&\leq C\,\norm{\Psi}_{L^1},
\end{align*}
where all norms are computed with respect to $g_0$ and where we use that $\Vol(M,g_0)\leq C(\gamma)$ and thus $\norm{\Om_j}_{L^1}\leq C\,\norm{\Om_j}_{L^2}=C$.
\end{proof}

In order to apply these results to the analysis of solutions of \eqref{eq.flow}, we note the following.
\begin{remark}\label{rem:Linfty-T}
Let $(g=e^{2u}g_0,\phi)$ be a solution of \eqref{eq.flow} whose energy density is bounded,
\begin{equation}\label{ass:energy-density-with-K} 
\sup_{x\in M, t\in [0,T)} \frac12 \Abs{d\phi(x,t)}_{g}^2\leq K
\end{equation}
for some $K\in\R$. Then 
\begin{equation}\label{est:uniform-Linfty-T}
\Norm{e^{-2u}\,\tf T(\phi,g_0)}_{L^\infty(M,g_0)}=
\Norm{\tf T(\phi,g)}_{L^\infty(M,g)}\leq C \sup_{M\times [0,T)} \Abs{d\phi(x,t)}_{g}^2 \leq C\cdot K,
\end{equation}
for a universal constant $C$.
\end{remark}

Combined with Lemma \ref{lemmaCkL1} and Corollary \ref{corA}, we thus obtain that $\pt g(t)$ is bounded uniformly with respect to the $C^k$ norm on $(M,g(t))$. In the following analysis of the conformal factor and the map component it will be important that $\pt g(t)$ is not just controlled with respect to this particular norm but that we can control the velocities at different times all with respect to the same norm, cf. also \cite{RT-global} and \cite{R-existence} for related problems. To this end we recall from \cite[Lemma 3.3] {RT-horizontal}
that the estimate

\begin{equation}\label{est:Ck-equiv-horiz}
\norm{\Om}_{C^k(M,g(t_1))}\leq \norm{\Om}_{C^k(M,g(t_2))}\cdot \exp{\left(\int_{t_1}^{t_2}\norm{\pt g(t)}_{C^k(M,g(t))} dt\right)}
\end{equation}
holds true for any fixed tensor $\Om$, any smooth curve of metrics $(g(t))_{t\in[0,T)}$ on $M$ and all $0\leq t_1<t_2<T$ with a constant $C$ that depends only on $k$ and the order of the tensor. Applied to our horizontal curve of metrics $g_0(\cdot)$ we obtain the following corollary.

\begin{cor}\label{corB}
Let $(g=e^{2u}g_0,\phi)$ be a solution of \eqref{eq.flow} on an interval $[0,T)$, $T<\infty$, and assume that 
\eqref{ass:energy-density-with-K} holds true for some $K\in \R$ as well as that 
$\inj(M,g_0(t))\geq \eps_0>0$, $t\in[0,T)$. Then the $C^k(M,g_0(t))$-norms, $t\in[0,T)$, are uniformly equivalent in the sense that for every $k_0\in \N_0$ there exists a constant 
$C=C(k,\eps_0,T,K,\gamma)$ so that 
\begin{equation}\label{est:equiv-norms}
\Norm{h}_{C^k(M,g_0(s))}\leq C\cdot \Norm{h}_{C^k(M,g_0(t))}, \quad \forall s, t\in [0,T),\, \forall h\in \Sym^2(T^*M).
\end{equation}
Furthermore, the velocity of $g_0$ is controlled with respect to every such $C^k$-norm
\begin{equation} \label{est:Ck-g}
\Norm{\pt g_0(t)}_{C^k(M,g_0(s))} \leq C, \quad \forall s, t \in [0,T),
\end{equation}
again with $C=C(k,\eps_0,T,K,\gamma)$.
\end{cor}

\begin{proof}
Let $(g=e^{2u(t)}g_0(t),\phi(t))$ be a solution of \eqref{eq.flow} as in the corollary. Given any $k\in\N_0$ and $t\in[0,T)$, we can use Corollary \ref{corA} to estimate
\begin{equation}\label{est:g-Ck-right-t}
\begin{aligned}
\Norm{\pt g_0(t)}_{C^k(M,g_0(t))}&=\Norm{P_{g_0}^\Hol(e^{-2u}\,\tf T(\phi,g))}_{C^k(M,g_0(t))}\\
&\leq C\cdot \Norm{e^{-2u}\,\tf T(\phi,g)}_{L^1(M,g_0(t))}\\
&\leq C \cdot \Norm{e^{-2u}\,\tf T(\phi,g)}_{L^\infty(M,g_0(t))}\;\leq C,
\end{aligned}
\end{equation}
where we applied Remark \ref{rem:Linfty-T} in the last step and where $C$ depends only on $\eps_0$, $k$, $K$ and as always the genus of $M$.\\

Applying \eqref{est:Ck-equiv-horiz} once for $t\mapsto g(t)$ and once to $t\mapsto g(-t)$, we thus obtain that 
for any fixed tensor $h\in \Sym^2(T^*M)$ 
\begin{equation}
\Norm{h}_{C^k(M,g_0(s))}\leq e^{CT}\Norm{h}_{C^k(M,g_0(t))} \text{ for all } s,t\in[0,T),
\end{equation}
with $C=C(k,\eps_0,\gamma, K)$, which yields \eqref{est:equiv-norms}. The second claim of the corollary is then an immediate consequence of \eqref{est:equiv-norms} and \eqref{est:g-Ck-right-t}.
\end{proof}

\begin{remark}\label{rem2CorB}
Corollary \ref{corB} implies in particular that in this setting the metrics $(g_0(t))_{t\in[0,T)}$ are uniformly equivalent and so are the corresponding Sobolev norms $\norm{\cdot}_{W^{k,p}(M,g_0(t))}$, where $t\in [0,T)$, for any $k \in \mathbb{N}_0$ and $p \in [1,\infty]$. This fact will be used several times in the Sections \ref{sec.struwe} and \ref{sec.curv}.
\end{remark}

The reason why we can apply the above results for solutions of \eqref{eq.flow} satisfying \eqref{ass:energy-density-with-K} is the following control on the injectivity radius which is based on results from \cite{RT-neg} that will be recalled in the proof below.

\begin{lemma}\label{lemma:inj-rad-g0}
Let $(g=e^{2u}g_0,\phi)$ be a solution of \eqref{eq.flow} on an interval $[0,T)$, $T<\infty$, for which 
\eqref{ass:energy-density} is satisfied. Then 
\begin{equation*}
\inf_{t\in [0,T)}\inj(M,g_0(t)) \geq \eps_0>0,
\end{equation*}
where the constant $\eps_0>0$ depends only on $\inj(M,g_0(0))$, $T$, $\gamma$ and an upper bound on the energy density $e(\phi,g)$.
\end{lemma}

\begin{proof}
Let us first assume that the genus of $M$ is $\gamma\geq 2$ and consequently that the metrics $g_0$ are hyperbolic.\\

We recall that for any $\delta\in (0,\arsinh(1))$ the $\delta$-thin part of a hyperbolic surface $(M,g_0)$ is given by the union of so-called collar neighbourhoods $\Col(\ell)$ around the simple closed geodesics of $(M,g_0)$ of length $\ell<2\delta$ and that each such collar $\Col(\ell)$ is isometric to $\big((-Y(\ell),Y(\ell))\times S^1, \varrho^2(s)(ds^2+d\theta^2)\big)$, where 
\begin{equation*}
\varrho(s)=\frac{\ell}{2\pi \cos(\frac{\ell s}{2\pi})} \qquad\text{ and }\qquad  Y(\ell)=\frac{2\pi}{\ell}\Big(\frac\pi2-\arctan\Big(\sinh\Big(\frac{\ell}{2}\Big)\Big) \Big).
\end{equation*}
We also recall that along a horizontal curve of metrics, $\pt g_0(t)=\Real(\Om(t))$, with $\Om(t)\in \newhol(M,g_0(t))$, 
the evolution of the length $ \ell(t)$ of the central geodesic of such a collar
is determined solely in terms of the principal part $b_0dz^2$ of the Fourier expansion of 
\begin{equation*}
\Om=\sum_{n\in\Z} b_ne^{ni\theta}e^{ns} dz^2, \quad z=s+i\theta, \, (s,\theta)\in (-Y(\ell),Y(\ell))\times S^1,
\end{equation*}
namely $\ddt \ell=-\frac{2\pi^2}{\ell}\Real(b_0)$, see e.g. \cite[Remark 4.12]{RT-neg}.\\

Furthermore, if $\pt g_0=\Real(P^\newhol_{g_0} ( \Psi(t)))$ is described as projection of a quadratic differential $\Psi(t)$ onto $\newhol(g_0)$ as in our situation, the results obtained in \cite{RT-neg} allow us to compute $\ddt \ell$ in terms of $\Psi(t)$. More precisely,  Lemma 2.2 of \cite{RT-neg} gives that
\begin{equation*}
\Aabs{\frac{d\ell}{dt}+\frac{\ell^2}{16\pi^3}\Real\Scal{\Psi,dz^2}_{L^2(\Col,g_0)}} \leq C\ell^2\Norm{\Psi}_{L^1(M,g_0)}
\end{equation*}
for $\ell\in (0,2\arsinh(1))$.\\

Since $\norm{dz^2}_{L^1(\Col,g_0)}=8\pi Y(\ell)\leq \frac{C}{\ell}$ and since the area of $(M,g_0)$ is determined in terms of the genus, this implies in particular that 
\begin{equation}\label{est:dell-Linfty}
-\frac{d\ell}{dt}\leq C\ell\,\Norm{\Psi}_{L^\infty(M,g_0)},
\end{equation}
for every closed geodesic of length $\ell\in (0,2\arsinh(1))$, in particular for the shortest simple closed geodesic of a hyperbolic surface with injectivity radius $\inj(M,g)=\half \ell\leq \arsinh(1)$.\\

Applied to solutions of \eqref{eq.flow} satisfying \eqref{ass:energy-density}, we obtain, using in particular also Remark \ref{rem:Linfty-T}, that 
the function $t\mapsto \inj(M,g_0(t))$, which is Lipschitz and thus differentiable almost everywhere, satisfies
\begin{equation*}
-\ddt \inj(M,g_0)\leq C\Norm{e^{-2u}\,\tf T(\phi,g)}_{L^\infty(M,g_0)}\cdot \inj(M,g_0) \leq C \cdot \inj(M,g_0)
\end{equation*}
for almost every time $t$ with $\inj(M,g_0(t))\leq \arsinh(1)$. This implies the claim of the lemma for surfaces of genus at least $2$.\\

We finally consider the case that $M$ is a torus. Here it is sufficient to observe that since we project $L^2$-orthogonally, we know that 
\begin{equation*}
\Norm{\pt g_0}_{L^2(M,g_0)}\leq \Norm{e^{-2u}\,\tf{T}(\phi,g)}_{L^2(M,g_0)}\leq C
\end{equation*}
is uniformly bounded. In particular, the curve $(g_0(t))_{t\in [0,T)}$ has finite $L^2$-length and thus its projection to Teichm\"uller space has finite length with respect to 
the Weil-Petersson metric. But the Teichm\"uller space of the torus (equipped with the Weil-Petersson metric) is isometric to the hyperbolic plane and so in particular complete, and thus the injectivity radius must remain bounded away from zero as described in the lemma.
\end{proof}

\section{Evolution of the conformal factor}\label{sec.struwe}
For two-dimensional Ricci Flow -- which evolves only in conformal directions -- regularity of the conformal factor can be proved from a uniform bound on its Liouville energy, see Struwe \cite{Struwe}. Motivated by this, we consider the Liouville energy of the conformal factor $u$ along the Harmonic Ricci Flow,
\begin{equation}\label{def.Liouville}
E_L(t):=\frac12\int_M \Big(\Abs{du(t)}_{g_0(t)}^2+2\bar{K}u(t)\Big) d\mu_{g_0(t)}.
\end{equation}
One of the main differences to the Ricci Flow case is that in our case the background metric $g_0$ is not fixed but is an evolving horizontal curve of metrics $g_0(t)$.\\

As observed in \cite{Struwe}, Jensen's inequality implies that  
\begin{equation}\label{eq.Jensenu}
2\fint_M u\, d\mu_{g_0}\leq \log\Big(\fint_M e^{2u}d\mu_{g_0} \Big)=\log\bigg(\frac{\text{vol}(M,g)}{\text{vol}(M,g_0)}\bigg)=0
\end{equation}
since the flow preserves the volume. In particular, for  surfaces of genus $\geq 1$,
\begin{equation}\label{est:H1-Liouville} 
\int_M\Abso{d u}^2d\mu_{g_0}\leq E_L.
\end{equation}
For the renormalised Ricci Flow, a short computation yields $\ddt E_L=-\int (K_g-\bar{K})^2d\mu_g$ (see \cite{Struwe}). In the case of Harmonic Ricci Flow, we do not obtain such a nice monotone behaviour, but we can still prove that $E_L$ is uniformly bounded along the flow. Thus \eqref{est:H1-Liouville} yields $H^1$-bounds for $u$ from which we can then also derive $H^2$-bounds.\\

In Section \ref{sec.curv}, the $H^2$-bounds obtained here will then be combined with estimates from the two previous sections, yielding higher regularity estimates and in particular the desired curvature bounds claimed in Proposition \ref{prop:main-result}.

\subsection{Evolution of the Liouville energy}
We start with an evolution inequality for the Liouville energy defined in \eqref{def.Liouville}.
\begin{prop}\label{prop.Liouville}
Let $(g(t)=e^{2u(t)}g_0(t),\phi(t))_{t\in[0,T)}$,  $T<\infty$, be a solution of \eqref{eq.flow} for which \eqref{ass:energy-density} holds. Then there exists a constant $C$, such that the Liouville energy satisfies
\begin{equation*}
\ddt E_L(t)\leq -\frac12\int_M (K_g-\bar{K})^2d\mu_g+C\cdot E_L(t)+C\cdot \Norm{\tau_{g}(\phi)}_{L^2(M,g)}^2 +C.
\end{equation*}
\end{prop}
\begin{proof}
Using \eqref{Gausseq}, \eqref{eq:u} and the fact that $\pt g_0$ is trace-free and thus $\pt d\mu_{g_0}=0$, we find that the Liouville energy evolves according to
\begin{equation}\label{est:Liouville-1}
\begin{aligned}
\ddt E_L(t)&=\int_M u_t\cdot (-\Delta_{g_0 }u+\bar{K})d\mu_{g_0}-\frac12\int_M \Scal{\pt g_0, du\otimes du}_{g_0}\dmo\\
&=\int_M K_g \cdot u_t\, d\mu_g+R_1(t)\\
&=-\int_M(K_g-\bar{K})^2 d\mu_g+\alpha\int_M K_g\cdot \big(e(\phi,g)-\bar E(t)\big)\,d\mu_g\\
&\quad\, +\int_M e^{2u}K_g\cdot \frac12 \varrho \,\dmo-\int_M e^{2u}K_g\cdot du(X)\,\dmo +R_1(t)\\
&= I_1(t) + I_2(t) + I_3(t) + I_4(t) + R_1(t),\phantom{\int}
\end{aligned}
\end{equation}
with an error term $R_1(t) = -\frac12\int_M \langle \pt g_0, du\otimes du\rangle_{g_0}\dmo$ that according to \eqref{est:g-Ck-right-t} and \eqref{est:H1-Liouville} is bounded by
\begin{equation*}
\Abs{R_1(t)}\leq \Norm{\pt g_0}_{L^\infty(M,g_0)}\int_M\Abso{du}^2\dmo\leq C\cdot E_L(t).
\end{equation*}
For the second term in \eqref{est:Liouville-1}, we use $\int_M \big( e(\phi,g)-\bar E(t)\big)d\mu_g=0$ and the assumption that the energy density is bounded, which yields the following estimate
\begin{align*}
I_2(t) &= \alpha\int_M K_g\cdot \big(e(\phi,g)-\bar E(t)\big)\,d\mu_g = \alpha\int_M \big(K_g-\bar{K}\big) \big(e(\phi,g)-\bar E(t)\big)\,d\mu_g\\
&\leq \frac12\int_M (K_g-\bar{K})^2 d\mu_g+C\int_M \big(e(\phi,g)-\bar{E}(t)\big)^2 d\mu_g\\
&= \frac12 \int_M (K_g-\bar{K})^2 d\mu_g + C\int_M e(\phi,g)^2 d\mu_g - C\bar{E}(t)^2\cdot \Vol(M,g)\\
&\leq \frac12 \int_M (K_g-\bar{K})^2 d\mu+C.
\end{align*}
Since $I_1(t) = -\int_M(K_g-\bar{K})^2 d\mu_g$, we thus obtain that 
\begin{equation*}
I_1(t)+I_2(t)\leq -\frac{1}{2} \int_M(K_g-\bar{K})^2 d\mu_g+C.
\end{equation*}

Using \eqref{Gausseq} as well as $\int_M\varrho\,\dmo=0$, we find for the third term in \eqref{est:Liouville-1} that
\begin{equation}
\begin{aligned}
I_3(t) &= \frac12 \int_M e^{2u}K_g\cdot \varrho\, \dmo = -\frac12 \int_M\Deltao u \cdot \varrho\, \dmo\\
&\leq \frac14\Big(\int_M\Abso{du}^2\dmo+\int_M\Abso{d\varrho}^2\dmo \Big)\\
&\leq C\int_M\Abso{du }^2\dmo+C\int_M\Abs{\tau_g(\phi)}^2d\mu_g +C\\
&\leq C\cdot E_L(t)+C\cdot \Norm{\tau_{g}(\phi)}_{L^2(M,g)}^2 +C
\end{aligned}
\end{equation}
using Lemma \ref{lemma:rho-X}. Finally, we write 
\begin{align*}
I_4(t) &= -\int e^{2u}K_g \, du(X)\, \dmo\\
&=-\bar{K}\int_M du(X)\,\dmo+\int_M\Delta_{g_0}u\cdot du(X)\,\dmo\\
&=J_1(t) + J_2(t),
\end{align*}
and observe that 
\begin{equation*}
\Abs{J_1(t)}\leq C\int_M\Abso{du}^2\dmo + C\Norm{X}_{L^2(M,g_0)}^2 \leq C\cdot E_L(t)+C,
\end{equation*}
again by Lemma \ref{lemma:rho-X}. The remaining integral $J_2(t)$ is the highest order term and thus requires a more careful estimate making use of its precise structure. For this reason, let $(f_\eps)$ be the flow generated by the vectorfield $X$. Since $du(X)=\peps (u\circ f_\eps)\vert_{\eps=0}$, we can write 
\begin{equation}
\begin{aligned} 
J_2(t) &=\int_M \Delta_{g_0}u\cdot  \peps(u\circ f_\eps) \dmo \big\vert_{\eps=0}=-\int \Scal{du,\peps(d(u\circ f_\eps))}_{g_0}\dmo\big \vert_{\eps=0} \\
&=-\frac{d}{d\eps}\frac12\int_M \Abs{d(u\circ f_\eps)}_{f_\eps^*g_0}^2d\mu_{f_\eps^*g_0}\big\vert_{\eps=0} \\
&\quad\, -\frac12\int_M \Big( \Scal{L_Xg_0, du\otimes d u}_{g_0}-\Abso{du}\frac12\tr_{g_0}(L_Xg_0)\Big)\dmo\\
&=-\frac12\int_M\Scal{\tf{L_Xg_0},du\otimes du}_{g_0} \; \dmo,
\end{aligned}
\end{equation}
where we use that 
$\peps\vert_{\eps=0}(\abs{du}_{f_\eps^* g_0}^2)=\peps((f_\eps^*g_0)^{ij}\pxi u\,\pxj u)=-g_0^{ik}g_0^{jl}(L_Xg_0)_{kl} \pxi u\,\pxj u$. Remark that it is no coincidence that the above expression depends only on the trace-free part and not the trace-part of $L_Xg_0$ but rather a consequence of the fact that the Dirichlet energy (here considered of $u$) is conformally invariant in dimension two. By construction 
\begin{equation*}
T(\phi,g)-\big(e^{2u}L_Xg_0+X(e^{2u})g_0\big)=\pt g=2\pt u \cdot e^{2u}g_0+e^{2u}\pt g_0,
\end{equation*}
so the trace-free part of $L_Xg_0$ is given by 
\begin{equation*}
\tf{(L_Xg_0)}=e^{-2u}\,\tf T(\phi,g)-\pt g_0.
\end{equation*}
Thus, by assumption \eqref{ass:energy-density} and the bound on $\pt g_0$ from \eqref{est:g-Ck-right-t}, we obtain that  
\begin{equation*}
\Abs{J_2(t)}\leq C\Norm{\tf{L_Xg_0}}_{L^{\infty}(M,g_0)}\int_M\Abso{du}^2\dmo\leq C\cdot E_L(t).
\end{equation*}
Plugging everything back into \eqref{est:Liouville-1} completes the proof of the proposition.
\end{proof}

\subsection{$H^1$-estimates for $u$ and other corollaries}
In view of \eqref{energy-decay}, an immediate consequence of Proposition \ref{prop.Liouville} is that the Liouville energy is uniformly bounded on $[0,T)$, e.g.~by Gr\"{o}nwall's Lemma. Moreover, it also implies that $\int_M(K_g-\bar{K})^2 d\mu_g$ and thus also $\int_M K_g^2d\mu_g$ is integrable in time. Hence, using also \eqref{est:H1-Liouville}, we obtain the following corollary of Proposition \ref{prop.Liouville}. 
\begin{cor}\label{cor:H1-u}
Let $(g,\phi)$ be a solution of \eqref{eq.flow} satisfying \eqref{ass:energy-density}, i.e. with bounded energy density on $M \times [0,T)$. Then 
\begin{equation} \label{H1-est}
\sup_{t\in[0,T)}\int_M\Abs{du}_{g_0(t)}^2d\mu_{g_0(t)}+\TMint \Abs{K_g}^2 d\mu_{g(t)}dt<\infty.
\end{equation}
\end{cor}
In the remainder of this subsection, following the ideas of Struwe \cite{Struwe}, we prove two further corollaries whose main point it is to show that we can integrate first spatial derivatives of $u$ multiplied with arbitrary powers of the conformal factor. This will be crucial when we prove $H^2$-estimates in the next subsection. The following is a consequence of the $H^1$-(semi-norm-)bound from Corollary \ref{cor:H1-u} and the Moser-Trudinger Inequality.
\begin{cor}\label{Cor:Moser-Trudinger}
Let $(g,\phi)$ be as above. Then
\begin{equation} \label{est:ut}
\TMint \Big(e^{-2u}\Abs{\Deltao u}^2+\Abs{u_t}^2e^{2u}\Big) \dmo\, dt<\infty.
\end{equation}
Furthermore, for any number $k\in\R$, we have
\begin{equation} \label{est:Moser-Trud} 
\sup_{t\in[0,T)}\int_M e^{k\abs{u}}\dmo <\infty
\end{equation}
and 
\begin{equation} \label{est:weighted-H1}
\TMint e^{ku}\Abso{du}^2 \dmo \, dt<\infty.
\end{equation}
\end{cor}
\begin{proof}
We first prove \eqref{est:Moser-Trud}. By the Moser-Trudinger Inequality, there exists $\beta_0>0$ such that
\begin{equation*}
\int_M e^{\beta_0\abs{v}^2} d\mu_{g_0} \leq C < \infty, \qquad \forall v\in H^1 \text{ with } \Norm{v}_{H^1(M,g_0)}\leq 1
\end{equation*}
and this is valid for all $t$ with the same constants $\beta_0$ and $C$, since the metrics $g_0(t)$ are uniformly equivalent. Since we only have a bound on the $H^1$ semi-norm of $u$, we combine this with the Poincar\'{e} inequality. More precisely, given any $t \in [0,T)$, we let
\begin{equation*}
v= \frac{(u-\bar{u})}{\norm{u-\bar{u}}_{H^1(M,g_0)}}, \quad \text{ where } \quad \Norm{u-\bar{u}}_{H^1(M,g_0)}\leq C\Norm{du}_{L^2(M,g_0)}\leq c_0.
\end{equation*}
Then from the Moser-Trudinger Inequality, setting $\beta = \beta_0 / c_0^2$, we obtain
\begin{equation}\label{eq.MTuubar}
\int_M e^{\beta\abs{u-\bar{u}}^2} d\mu_{g_0} \leq C.
\end{equation}
Combining \eqref{eq.MTuubar} with $\int_M e^{2u} d\mu_{g_0} = \int_M d\mu_g = V_0$, we obtain
\begin{align*}
V_0 \cdot e^{-2\bar{u}} &= e^{-2\bar{u}}\int_M e^{2u} d\mu_{g_0} = \int_M e^{2(u-\bar{u})} d\mu_{g_0}\\
&\leq \int_M e^{\beta\abs{u-\bar{u}}^2+C_\beta} \, d\mu_{g_0} = e^{C_\beta} \int_M e^{\beta\abs{u-\bar{u}}^2} d\mu_{g_0}  \leq C.
\end{align*}
Hence $e^{-2\bar{u}}$ is uniformly bounded from above, or equivalently $\bar{u}$ is uniformly bounded from below, $\bar{u} \geq -c_1>-\infty$. In view of \eqref{eq.Jensenu}, we also know that $\bar{u}\leq 0$, and hence $\abs{\bar{u}}\leq c_1$ for all $t\in[0,T)$. Finally, using \eqref{eq.MTuubar} once more, we find
\begin{equation*}
C \geq \int_M e^{\beta\abs{u-\bar{u}}^2} d\mu_{g_0} \geq \int_M e^{\frac{\beta}{2}\abs{u}^2 - \beta\abs{\bar{u}}^2} d\mu_{g_0} \geq \int_M e^{\frac{\beta}{2}\abs{u}^2}d\mu_{g_0} \cdot e^{-\beta c_1^2}
\end{equation*}
or equivalently 
\begin{equation*}
\int_M e^{\frac{\beta}{2}\abs{u}^2}d\mu_{g_0} \leq C\cdot c^{\beta c_1^2}
\end{equation*}
which yields the bound claimed in \eqref{est:Moser-Trud}.\\

Next, to prove the first part of \eqref{est:ut}, we note that by \eqref{Gausseq}
\begin{equation}\label{doubleintegraleq}
\begin{aligned}
\TMint e^{-2u}\Abs{\Deltao u}^2 \dmo \, dt &= \TMint e^{-2u}\big({-e^{2u}}K_g+\bar{K}\big)^2 \dmo \, dt\\
&\leq C \TMint \Big(e^{-2u}e^{4u}K_g^2+e^{-2u}\bar{K}^2\Big) \dmo \, dt\\
&\leq C \TMint K_g^2\, d\mu_g \, dt + C\bar{K}^2 \TMint e^{-2u} d\mu_{g_0} \, dt,\\
&\leq C <\infty
\end{aligned}
\end{equation}
where the last two steps follows from \eqref{H1-est} and \eqref{est:Moser-Trud}.\\

Before we show the second estimate of \eqref{est:ut}, we now derive the bound \eqref{est:weighted-H1} on weighted integrals of $du$, where in view of Corollary \ref{cor:H1-u}, we can assume that $k\neq 0$. We rewrite 
\begin{equation}
\begin{aligned} 
\int_Me^{ku}\Abso{du}^2 \dmo&=\frac1k\int_M \Scal{d(e^{ku}),du}_{g_0}\dmo=-\frac1k \int_M \Deltao u\cdot e^{ku} \dmo\\
&\leq \int_M \Abs{\Deltao u}^2 e^{-2u}\dmo+C\int_M e^{2(k+1)u}\dmo.
\end{aligned}
\end{equation}
From this formula, applying the estimates \eqref{est:Moser-Trud} and \eqref{doubleintegraleq}, which we have already proven, we obtain \eqref{est:weighted-H1}.\\

Finally, from the evolution equation \eqref{eq:u}, we have
\begin{align*}
\TMint e^{2u}\Abs{u_t}^2\, \dmo \, dt  &\leq C \TMint e^{-2u}\Abs{\Delta_{g_0}u}^2 \dmo \, dt\\
&\quad\, +C\TMint \bar{K}^2\big(1-e^{-2u}\big)^2 e^{2u} \dmo \, dt\\
&\quad\,+ C\TMint \alpha^2 e(\phi,g_0)^2 e^{-2u} \dmo \, dt\\
&\quad\, +C\TMint \Big(\varrho^2 +\Abs{du(X)}^2\Big) e^{2u} \dmo \, dt.
\end{align*}
While we have bounded the first term on the right hand side in \eqref{doubleintegraleq}, all the other terms are bounded as well, in view of \eqref{est:Moser-Trud}, \eqref{est:weighted-H1}, the bound on the energy density, and the bounds on $X$ and $\varrho$ in Lemma \ref{lemma:rho-X}, so this yields the second estimate of \eqref{est:ut}.
\end{proof}

For the next corollary, as well as in several places in the next subsection, we will use the standard interpolation inequality 
\begin{equation} \label{est:stand-interpol}
\Norm{f}_{L^4}^2\leq C\Norm{f}_{L^2}\cdot \Norm{f}_{H^1}
\end{equation}
and variants of it. The inequality \eqref{est:stand-interpol} follows from the Sobolev embedding $W^{1,1}\hookrightarrow L^2$ applied to $f^2$. Our last corollary here is the following.
\begin{cor}\label{Cor:W14-u}
Let $(g,\phi)$ be as above. Then for any $k\in\R$ there exists a function $f\in L^1([0,T])$ such that 
\begin{equation*}
\int_M\Abso{du}^4e^{2ku}\dmo\leq f(t)\cdot \Big(\int_M\Abs{ \Deltao u}^2\dmo +1\Big).
\end{equation*}
\end{cor}
\begin{proof}
Using that the Sobolev embeddings $W^{1,1}(M,g_0(t))\hookrightarrow L^2(M,g_0(t))$ are uniformly bounded for $t\in[0,T)$, we estimate
\begin{equation}
\begin{aligned}
\int_M\Abso{du}^4e^{2ku}\dmo&=\Norm{\abso{du}^2e^{ku}}_{L^2(M,g_0)}^2 \leq C \Norm{\abso{du}^2e^{ku}}_{W^{1,1}(M,g_0)}^2\\
&= C\cdot \Big(\int_M e^{ku}\Big(\Abso{du}^2+\Abso{du}^3+\Abso{\na_{g_0}du}\cdot \Abso{du} \Big)\dmo\Big)^2\\
&\leq C \int_M\Abso{du}^2e^{2ku}\dmo \cdot \int_M \Big(\Abso{du}^2 +\Abso{du}^4+\Abso{\nao du}^2\Big)\dmo.
\end{aligned}
\end{equation}

Recall from Corollary \ref{cor:H1-u} that $\int_M \abso{du}^2 \dmo \leq C$ is bounded uniformly, and from Corollary \ref{Cor:Moser-Trudinger} that the function $f_k:=\int_M\abso{du}^2e^{2ku}\dmo$ is integrable in time. Furthermore, we have
\begin{equation}\label{eq.HessLap}
\int_M \Abso{\nao du}^2\dmo\leq \int_M \Big(\Abs{\Deltao u}^2 +C\cdot\Abso{du}^2\Big)\dmo \leq C\Big(\int_M \Abs{\Deltao u}^2 \dmo + 1\Big),
\end{equation}
where $C$ comes from the curvature of $g_0$ (and is thus obviously bounded). Finally, by \eqref{est:stand-interpol}, we also find
\begin{equation*}
\int_M \Abso{du}^4\dmo\leq C \int_M \Abso{du}^2 \dmo \cdot \int_M \Big(\Abso{du}^2+\Abso{\nao du}^2\Big)\dmo \leq C\Big(\int_M \Abs{\Deltao u}^2 \dmo + 1\Big).
\end{equation*}
This concludes the argument.
\end{proof}

\subsection{$H^2$-estimates for $u$}
In this subsection, we follow the overall argument of \cite{Struwe}, with the necessary modifications due to the more complicated evolution equations, to derive $H^2$-estimates for the conformal factor, making use of the three corollaries of the last subsection. The goal is to prove the following.
\begin{prop}\label{prop:H2-u}
Let $(g,\phi)$ be a solution of \eqref{eq.flow} satisfying \eqref{ass:energy-density}, i.e. with bounded energy density on $M \times [0,T)$. Then, we have
\begin{equation} \label{H2-est}
\sup_{t\in[0,T)}\int_M\Abs{\Deltao u}^2d\mu_{g_0(t)}<\infty
\end{equation}
and thus
\begin{equation} \label{H2-est2}
\sup_{t\in[0,T)}\Norm{u(t)}_{H^2(M,g_0(t))}<\infty.
\end{equation}
\end{prop}
\begin{proof}
As in \cite{Struwe}, we multiply the equation for the conformal factor, in our case 
\begin{equation}\label{testedevoleq}
u_t-e^{-2u}\Deltao u=\bar{K}(1-e^{-2u})+\alpha\cdot\big(e(\phi,g_0)e^{-2u}-\bar E(t)\big)+\frac12\varrho-du(X)
\end{equation}
with $-\Deltao u_t\cdot e^{2u}$ and integrate over $(M,g_0)$. Using in particular that $\pt g_0$ is trace-free and thus $\pt d\mu_{g_0}=0$, we first observe that the resulting integral on the left hand side is
\begin{equation}\label{4.17}
\begin{aligned} 
I_1(t)&:=\int_M\Abso{d u_t}^2e^{2u} \dmo+2\int_M u_t\Scal{du, d u_t}_{g_0} e^{2u} \dmo +\frac12\ddt\int_M\Abs{\Deltao u}^2\dmo\\
&\quad\,-\frac12 \peps \int_M\Abs{\Delta_{g_0(t+\eps)} u}^2 \dmo\big\vert_{\eps=0}\\
&\geq \frac34 \int_M \Abso{du_t}^2e^{2u}\dmo-4\int_M \Abso{du}^2\cdot \Abs{u_t}^2 e^{2u}\dmo+\frac12\ddt  \int_M\Abs{\Deltao u }^2\dmo  \\
&\quad\; -C\Norm{\pt g_0}_{C^1}\int_M \Big(\Abso{\nao du}^2+\Abso{d u}^2\Big)  \dmo.
\end{aligned}
\end{equation}

On the other hand, multiplying the right hand side of \eqref{testedevoleq} with $-\Deltao u_t\cdot e^{2u}$ and integrating over $(M,g_0)$, we obtain that
\begin{align*}
I_1(t)&=\int_M (-\Deltao u_t) \Big(\bar{K}(e^{2u}-1)+\alpha\big(e(\phi,g_0)-\bar{E}(t)e^{2u}\big)+\frac12 \varrho\, e^{2u} - du(X)e^{2u}\Big)\dmo\\
&\leq \frac14\int_M \Abso{d u_t}^2e^{2u} \dmo +C \int_M (\bar{K}^2+\alpha^2 \bar E(t)^2)\cdot \Abso{d(e^{2u})}^2 e^{-2u}\dmo\\
&\quad\;+C\alpha^2\int_M \Abso{d(e(\phi,g_0))}^2 e^{-2u}\dmo+C\int_M\Abso{d(\varrho e^{2u})}^2e^{-2u}\dmo\\
&\quad\;+C\int_M\Abso{d(du(X)e^{2u})}^2e^{-2u}\dmo.
\end{align*}
Using that $\alpha$, $\bar{K}$, $\bar{E}(t)$, $\norm{\pt g_0}_{C^1}$, and also $\int_M \abso{du}^2 \dmo$ are uniformly bounded and can thus be absorbed into the constants, we find, by plugging this back into \eqref{4.17},  
\begin{equation}\label{est:H2-1}
\begin{aligned}
I_2(t) &:= \int_M \Abso{du_t}^2e^{2u}\dmo+\ddt\int_M \Abs{\Deltao u }^2\dmo\\
&\leq 8\int_M \Abso{du}^2\cdot \Abs{u_t}^2 e^{2u}\dmo+C\int_M \Abso{\nao du}^2 \dmo\\
&\quad\;+C\int_M\big(\Abso{du}^4+1\big)\big(1+e^{2u}+e^{4u}+\Abs{X}^2e^{2u}\big)\dmo\\
&\quad\;+\int_M\big(\varrho^4 +\Abso{\nao X}^4\big) \dmo + C\cdot (J_1(t)+J_2(t)+J_3(t))
\end{aligned}
\end{equation} 
where
\begin{equation}\label{eq.Js}
\begin{aligned} 
J_1(t)&:=\int_M\Abso{d(e(\phi,g_0)}^2e^{-2u}\dmo\\
J_2(t)&:=\int_M\Abso{d\varrho}^2e^{2u}\dmo\\
J_3(t)&:=\int_M\Abso{\nao du}^2\Abs{X}^2 e^{2u}\dmo.
\end{aligned}
\end{equation}

Similar to \cite{Struwe}, the first term on the right hand side of \eqref{est:H2-1} can be treated using the interpolation inequality \eqref{est:stand-interpol}. More precisely, using Corollary \ref{cor:H1-u} and Corollary \ref{Cor:Moser-Trudinger} we obtain that 
\begin{align*} 
8\int_M \Abso{du}^2\cdot \Abs{u_t}^2 e^{2u}\dmo &\leq 8\Norm{(e^{u})_t}_{L^4(M,g_0)}^2 \Norm{du}_{L^4(M,g_0)}^2\\
&\leq C \Norm{(e^{u})_t}_{L^2(M,g_0)} \Norm{(e^{u})_t}_{H^1(M,g_0)}\Norm{du}_{L^2(M,g_0)} \Norm{du}_{H^1(M,g_0)}\\
&\leq \frac 1{16}\Norm{(e^{u})_t}_{H^1(M,g_0)}^2+ f(t)\cdot \Big(\int_M\Abso{\nao du}^2\dmo+1\Big)\\
&\leq \frac18\int_M\Abso{du_t}^2e^{2u}\dmo+\frac18 \int_M \Abso{du}^2\cdot \Abs{u_t}^2 e^{2u}\dmo\\
&\qquad\, +f(t)\cdot \Big(\int_M\Abs{\Deltao u}^2\dmo+1\Big),
\end{align*}
where here and in the following $f(t)$ denotes a generic function (allowed to change from line to line) that is integrable over $[0,T]$. In the particular case above we may actually choose $f(t)=C\int_M\abs{u_t}^2e^{2u}\dmo$, compare Corollary \ref{Cor:Moser-Trudinger}. Note also that we used \eqref{eq.HessLap} in the last step. Using absorption, the above becomes
\begin{equation}\label{est:H2-proof1}
8\int_M \Abso{du}^2\cdot \Abs{u_t}^2 e^{2u}\dmo \leq \frac14\int_M\Abso{du_t}^2e^{2u}\dmo +f(t)\cdot \Big(\int_M\Abs{\Deltao u}^2\dmo+1\Big).
\end{equation}

The second term on the right hand side of \eqref{est:H2-1} can be treated with \eqref{eq.HessLap}, while the next term can be bounded by
\begin{equation}
\big(\Norm{X}_{L^\infty}^2 + 1\big)\Big(C+\int_M (1+e^{4u})\Abs{du}^4 \dmo \Big)
\leq f(t)\cdot \Big(\int_M\Abs{\Deltao u}^2\dmo+1\Big),
\end{equation}
using Corollary \ref{Cor:W14-u}. Finally, by Lemma \ref{lemma:rho-X}, the fourth term, involving $L^4$-norms of $\varrho$ and $X$, is uniformly bounded and thus \eqref{est:H2-1} becomes
\begin{equation} \label{est:H2-2}
\begin{aligned}
\frac34 \int_M \Abso{du_t}^2e^{2u}\dmo+\ddt \int_M\Abs{\Deltao u }^2\dmo &\leq f(t)\cdot \Big(\int_M\Abs{\Deltao u}^2\dmo+1\Big)\\
&\quad\, + C\cdot \big(J_1(t)+J_2(t)+J_3(t)\big)
\end{aligned}
\end{equation}
for some $f\in L^1([0,T])$. In the following, we will estimate the terms $J_1$, $J_2$, $J_3$, given in \eqref{eq.Js}.\\

\textbf{Estimate for $J_1(t)$.} Using assumption \eqref{ass:energy-density} as well as 
\eqref{est:Moser-Trud} we may estimate
\begin{equation}\label{est:J1-1}
\begin{aligned}
J_1(t)&=\frac{1}{2}\int_M\Abso{d(\Abso{d \phi}^2)}^2 e^{-2u}\dmo=2 \int_M \Abso{\nao d\phi}^2\cdot\Abso{d \phi}^2e^{-2u}\dmo\\
&\leq C\int_M\Abso{\nao d\phi}^2\dmo\leq C\int_M\Abs{\Deltao \phi}^2\dmo+C\int_M\Abso{d\phi}^2 \dmo\\
&\leq C\int_M\Abs{\tau_{g_0} \phi}^2\dmo+C\int_M e(\phi,g_0)^2 \dmo +C \leq C\int_M\Abs{\tau_{g_0} \phi}^2\dmo+C.
\end{aligned}
\end{equation}
We stress that here the tension and the corresponding $L^2$ integral are to be computed on the constant curvature surface $(M,g_0)$ and not on $(M,g)$, so that the above term cannot be controlled by the evolution equation of the energy \eqref{energy-decay}. Instead, we estimate 
\begin{equation}\label{est:tau1}
\begin{aligned}
\int_M\Abs{\tau_{g_0} \phi}^2\dmo &= \int_M e^{4u} \Abs{\tau_{g} \phi}^2\dmo\\
&\leq \int_M\Abs{\phi_t}^2\cdot e^{4u}\dmo+C \int_M\Abs{X}^2 \Abso{d\phi}^2e^{4u}\dmo\\
&\leq \int_M \Abs{\phi_t}^2\cdot e^{4u}\dmo+C \int_M e(\phi,g) e^{6u}\dmo\\
&\leq \int_M \Abs{\phi_t}^2\cdot e^{4u}\dmo+C
\end{aligned}
\end{equation}
and bound the resulting weighted integral of $\phi_t$ by testing the evolution equation \eqref{eq:phi} with $\phi_te^{4u}$. This yields
\begin{align*}
\int_M e^{4u} \Abs{\phi_t}^2 \dmo &=\int_M e^{2u}\Deltao \phi \cdot \phi_t\, \dmo - \int_M d\phi(X)\cdot e^{4u} \phi_t\, \dmo\\
&= - \int_M e^{2u}\Scal{d\phi,d\phi_t}_{g_0} \dmo - 2\int_M e^{2u}\Scal{du,d\phi}_{g_0}\phi_t\, \dmo\\
&\quad\, - \int_M d\phi(X)\cdot e^{4u} \phi_t\, \dmo,
\end{align*}
which allows us to estimate
\begin{align*}
\int_M e^{4u} \Abs{\phi_t}^2 \dmo &\leq - \frac12 \ddt \int_M e^{2u}\Abso{d\phi}^2 \dmo + \int_M u_t \cdot \Abso{d\phi}^2 e^{2u}\,\dmo\\
&\quad\, +\Norm{\pt g_0}_{L^\infty(M,g_0)} \cdot\int_M \Abso{d\phi}^2 e^{2u}\,\dmo + \frac12 \int_M \Abs{\phi_t}^2e^{4u} \dmo\\
&\quad\, +C\int_M\Abso{d\phi}^2\Abso{du}^2 \dmo + C\Norm{X}_{L^\infty(M,g_0)}^2\cdot \int_M \Abso{d\phi}^2e^{4u} \dmo.
\end{align*}
Using that $\abso{d\phi}^2e^{-2u} \leq C$, as well as $\norm{\pt g_0}_{L^\infty(M,g_0)} \leq C$ and $\norm{X}_{L^\infty(M,g_0)}^2 \leq C$, we thus have
\begin{align*}
\int_M e^{4u} \Abs{\phi_t}^2 \dmo &\leq - \frac12 \ddt \int_M e^{2u}\Abso{d\phi}^2 \dmo + C\int_M u_t \cdot e^{4u}\,\dmo\\
&\quad\, + C\int_M e^{4u}\,\dmo + \frac12 \int_M \Abs{\phi_t}^2e^{4u} \dmo\\
&\quad\, +C\int_M e^{2u}\Abso{du}^2 \dmo + C \int_M e^{6u} \dmo.
\end{align*}
Hence
\begin{align*}
\int_M\Abs{\phi_t}^2e^{4u} \dmo+\ddt \int_M \Abso{d\phi}^2e^{2u}\dmo &\leq \int_M\Abs{u_t}^2 e^{2u}\dmo\\
&\quad\,+ C\int_M\big(\Abso{du}^2+1\big)\big(e^{2u}+e^{6u}\big)\dmo,
\end{align*}
where we remark that the right hand side is integrable over $[0,T]$ thanks to Corollary \ref{Cor:Moser-Trudinger}. Inserting this into \eqref{est:tau1}, we thus obtain
\begin{equation}\label{est:tau-for-later}
\TMint \Abs{\tau_{g_0} \phi}^2\dmo\leq C+\int_{M\times \{t=0\}}\Abso{d\phi}^2e^{2u}\dmo\leq C<\infty,
\end{equation}
which, once combined with \eqref{est:J1-1}, allows us to conclude that 
\begin{equation} \label{est:I}
J_1(t) \leq  C\cdot f_1(t)
\end{equation}
for a function $f_1\in L^1([0,T]).$\\

\textbf{Estimate for $J_2(t)$.} As $\varrho$ is a solution of \eqref{eq:rho}, we can estimate 
\begin{equation}\label{est:J2-1}
\begin{aligned}
\int_M\Abso{d\varrho}^2e^{2u}\dmo&\leq-\int_M \Deltao \varrho\cdot \varrho\, e^{2u}\dmo+C\int_M\Abso{d\varrho}\cdot \varrho\,\Abso{du}e^{2u}\dmo\\
&\leq \int_M \deltao\deltao\big(e^{-2u}\tf{T}(\phi,g)\big)\cdot \varrho\, e^{2u}\dmo + 2\bar{K}\int_M \varrho^2 e^{2u} \dmo \\
&\quad\, +\frac14\int_M\Abso{d\varrho}^2e^{2u}\dmo +\int_M\Abs{d u}^4 e^{4u} \dmo+\Norm{\varrho}_{L^4(M,g_0)}^4\\
&\leq \frac12\int_M\Abso{d\varrho}^2e^{2u}\dmo +C\int_M \Abso{\deltao(e^{-2u}\tf{T}(\phi,g))}^2e^{2u}\dmo\\ 
&\quad\, +f(t)\cdot \Big(\int_M\Abs{\Deltao u}^2\dmo+1\Big)
\end{aligned}
\end{equation}
where we applied Lemma \ref{lemma:rho-X} and Corollary \ref{Cor:W14-u} as well as $\bar{K}\leq 0$. The first term can be absorbed. To treat the second, we recall that $\tf{T}(\phi,g_0)=\tf{T}(\phi,g)$ and use \eqref{est:divT} to estimate
\begin{align*}
\int_M \Abso{\deltao(e^{-2u}\tf{T}(\phi,g))}^2e^{2u}\dmo &\leq \int_M \Abso{du}^2e^{2u}\dmo + C\int_M\Abs{\tau_{g}(\phi)}^2 e^{4u}\dmo\\
&\leq f(t) + C\int_M \Abs{\tau_{g_0}(\phi)}^2 \dmo,
\end{align*}
where $f\in L^1$, compare Corollary \ref{Cor:Moser-Trudinger}. Plugging this back into \eqref{est:J2-1} and using also \eqref{est:tau-for-later}, we conclude that
\begin{equation}  \label{est:II} 
J_2(t) \leq f_2(t)\cdot \Big(\int_M\abs{\Deltao u}^2\dmo+1\Big)
\end{equation}
for a function $f_2$ that is integrable in time.\\

\textbf{Estimate for $J_3(t)$.} Thanks to the bounds on $X$ from Lemma \ref{lemma:rho-X}, we have
\begin{equation*}
J_3(t)=\int_M\Abso{\nao du}^2\Abso{X}^2e^{2u}\dmo\leq C\int_M\Abso{\nao du}^2e^{2u}\dmo,
\end{equation*}
which in turn is bounded by
\begin{align*}
\int_M\Abso{\nao du}^2e^{2u}\dmo&\leq \int_M\Abs{\Deltao u}^2e^{2u}\dmo +C\int_M\Abso{du}^2e^{2u}\dmo\\
&\quad\, +C\int_M\Abso{du}^4e^{2u}\dmo + \frac12 \int_M\Abso{\nao du}^2e^{2u}\dmo.
\end{align*}
Thus, using absorption as well as the Corollaries \ref{Cor:Moser-Trudinger} and \ref{Cor:W14-u}, we deduce
\begin{equation}
J_3(t) \leq C\int_M\Abs{\Deltao u}^2e^{2u}\dmo+f(t)\cdot \Big(\int_M\Abs{\Deltao u}^2\dmo+1\Big).
\end{equation}
Observe that the first integral contains a different power of the conformal factor and hence needs to be further rewritten and estimated. Using the evolution equation \eqref{eq:u} once more, we get that for any $\eps>0$
\begin{equation}
\label{est:DeltauL2}
\begin{aligned} 
\int_M\Abs{\Deltao u}^2 e^{2u}\dmo&\leq \int_M\Deltao u\cdot u_t\,e^{4u}\dmo+R(t)\\
&\leq \eps\cdot \int_M\Abs{du_t}^2e^{2u}\dmo+C_\eps\int_M\Abso{du}^2e^{6u}\dmo\\
&+C\int_M \Abs{u_t}^2e^{2u}\dmo+ \int_M\Abso{du}^4e^{6u}\dmo+R(t).
\end{aligned}
\end{equation}
Here, $R(t)$ contains all ``lower order'' terms resulting from testing \eqref{eq:u} with $\Deltao u e^{4u}$. In particular, $R(t)$ can be estimated by 
\begin{equation}
\begin{aligned} 
R(t)&\leq \frac12\int_M\Abs{\Deltao u}^2 e^{2u}\dmo\\
&\quad\, +C\int_M e^{6u}\Big( (1-e^{-2u})^2+(e(\phi,g_0)e^{-2u}-\bar E)^2+\varrho^2+\Abso{du}^2\Abso{X}^2\Big)\dmo\\
&\leq \frac12 \int_M \Abs{\Deltao u}^2 e^{2u}\dmo +C\int_M e^{12u}\dmo +\Norm{\varrho}_{L^4}^4+C\int_M\Abso{du}^4\dmo +C,
\end{aligned}
\end{equation}
where we remark that the first term can be absorbed into the left hand side of \eqref{est:DeltauL2}.
In total, we thus obtain that for every $\eps>0$ there is a constant $C_\eps$ so that  
\begin{equation} \label{est:III}
J_3(t)\leq \eps\cdot \int_M\Abs{du_t}^2e^{2u}\dmo+ f(t)\cdot \Big(\int_M\Abs{\Deltao u}^2\dmo+C_\eps\Big).
\end{equation}

\textbf{Putting everything together.} Inserting \eqref{est:I}, \eqref{est:II} and \eqref{est:III} into \eqref{est:H2-2} and choosing $\eps>0$ sufficiently small we conclude that $t\mapsto \int_M\abs{\Deltao u }^2$ satisfies an estimate of the form
\begin{equation*}
\ddt \int_M\Abs{\Deltao u }^2\dmo\leq f(t)\cdot\Big(\int_M\Abs{\Deltao u}^2\dmo+1\Big), \quad f\in L^1([0,T])
\end{equation*}
and is thus uniformly bounded on $[0,T)$ e.g. by Gr\"{o}nwall's Lemma. Combined with the $H^1$-estimates we already proved, we thus have uniform bounds on the $H^2$-norm of $u$ as claimed, finishing the proof of Proposition \ref{prop:H2-u}. 
\end{proof}

\section{Higher regularity and curvature estimates}\label{sec.curv}

In this section, we finish the proof of Proposition \ref{prop:main-result}, using bootstrapping arguments as well as the following facts from the previous sections.\\

First, we know from Corollary \ref{corB} and Remark \ref{rem2CorB} that the metrics $g_0(t)$ and their $C^k$- and Sobolev-norms are uniformly equivalent for $t\in[0,T)$. We will thus compute all the norms in this section with respect to a \emph{fixed} metric $g_0(t_0)$, say for example $g_0(0)$. Moreover, these results also show that $t \mapsto g_0(t)$ is Lipschitz continuous with respect to any $C^k$-norm in space.\\

Second, we know from Lemma \ref{lemma:rho-X} that for every $1 \leq p<\infty$
\begin{equation}\label{5.rhoX}
\sup_{t\in[0,T)}\Norm{\varrho(t)}_{L^p(M)} + \Norm{X(t)}_{W^{1,p}(M)} < \infty.
\end{equation}
Third, by Proposition \ref{prop:H2-u},
\begin{equation}\label{5.H2u}
\sup_{t\in[0,T)}\Norm{u(t)}_{H^2(M)}<\infty,
\end{equation}
and finally, by Lemma \ref{lemma:inj-rad-g0}, we have
\begin{equation}\label{5.injg0}
\inf_{t\in [0,T)}\inj(M,g_0(t)) \geq \eps_0>0,
\end{equation}
Based on these facts, Proposition \ref{prop:main-result} will now be proved as follows. We first prove H\"{o}lder continuity of $u$ on $M\times[0,T]$. Then, we prove H\"{o}lder continuity for $\phi$, $d\phi$, $\varrho$, and $X$. In a last step, we then use this to obtain higher order ($C^{2,1;\beta}_{par}$) estimates for $u$ that will imply the curvature bounds claimed in Proposition \ref{prop:main-result}. Finally, we note that these curvature bounds together with \eqref{5.injg0} prove the injectivity radius bound claimed in Proposition \ref{prop:main-result}

\subsection{H\"{o}lder continuity}
Let $(g(t)=e^{2u(t)}g_0(t),\phi(t))_{t\in[0,T)}$,  $T<\infty$, be a solution of \eqref{eq.flow} for which \eqref{ass:energy-density} holds. By \eqref{5.H2u} and the Sobolev embedding theorem, for all $p\in[1,\infty)$ we have
\begin{equation}\label{5.Sobu}
\sup_{t\in[0,T)}\Norm{u(t)}_{L^\infty(M)} + \sup_{t\in[0,T)}\Norm{du(t)}_{L^p(M)} < \infty,
\end{equation}
where -- as explained above -- we compute the norms with respect to $g_0(0)$. Combined with Corollary \ref{corB} and Remark \ref{rem2CorB}, \eqref{5.Sobu} allows us to view $\pt - e^{-2u}\Deltao$ as a uniformly parabolic operator with uniformly bounded (even H\"{o}lder continuous) coefficients on the fixed Riemannian surface $(M,g_0(0))$. This allows us to apply standard parabolic theory.\\

In a first step, observe that
\begin{equation*}
\Abs{(\pt - e^{-2u}\Deltao)(u)} \leq C\big(1+\Abs{\varrho}+\Abs{X}+\Abs{du}\big),
\end{equation*}
which follows from \eqref{eq:u} and the fact that -- as a consequence of \eqref{ass:energy-density} and \eqref{5.Sobu} -- we also know that $\abs{d\phi}=\abs{d\phi}_{g_0(0)} \leq C\abs{d\phi}_{g_0(t)} = Ce^{u}\abs{d\phi}_g \leq C$ is uniformly bounded. Moreover, according to \eqref{5.Sobu} and Lemma \ref{lemma:rho-X}, the right hand side is in $L^p(M\times[0,T))$ for every $p\in[1,\infty)$. Thus for each such $p$ and every $\delta>0$, we obtain
\begin{equation*}
\nabla_{g_0(0)} du,\, \pt u \in L^p(M \times [\delta,T)),
\end{equation*}
and as $u$ is smooth on $[0,T)$ and the metrics $g_0(t)$ and their Sobolev norms are uniformly equivalent, in particular
\begin{equation}
\nabla_{g_0(0)}d u,\, \nabla_{g_0}d u,\, \pt u \in L^p(M \times [0,T)).
\end{equation}
Applying the Sobolev embedding theorem once more, we conclude that $u$ can be (continuously) extended to $M \times [0,T]$ with
\begin{equation}\label{5.Hoelderu}
u \in C^{0,\beta}(M \times [0,T]), \qquad \text{ for all } 0\leq \beta < 1,
\end{equation}
$C^{0,\beta}$ denoting the space of $\beta$-H\"{o}lder continuous functions with respect to the product distance $d = (d_{g_0(0)}^2 + d_{\mathbb{R}}^2)^{1/2}$.\\

Next, we note that by \eqref{eq:phi} and Lemma \ref{lemma:rho-X}
\begin{equation*}
\Abs{(\pt - e^{-2u}\Deltao)(\phi)} \leq C\big(e^{-2u}\Abs{d\phi}^2 + \Abs{X}\cdot\Abs{d\phi}\big) \leq C < \infty,
\end{equation*}
so similar to the above, we obtain
\begin{equation}\label{5.Lpphi}
\nabla_{g_0(0)} d\phi, \nabla_{g_0(t)} d\phi,\, \pt \phi \in L^p(M \times [0,T)), \qquad 1\leq p<\infty.
\end{equation}
Differentiating \eqref{eq:phi} in space then yields
\begin{align*}
\Abs{(\pt -e^{-2u}\Deltao)(\nabla_{g_0}\phi)} &\leq C\Abs{d\phi}\big(1+\Abs{d\phi}^2 + \Abs{\nabla_{g_0} d\phi} + \Abs{\nabla_{g_0} X}\big)\\
&\quad\, + C\Abs{X}\cdot\Abs{\nabla_{g_0} d\phi} + C\Abs{du}\big(\Abs{d\phi}^2 + \Abs{\nabla_{g_0} d\phi}\big)\\
&\leq C\big(1+\Abs{du}^2+\Abs{\nabla_{g_0} d\phi}^2 + \Abs{\nabla_{g_0}X}\big)
\end{align*}
with a right hand side that is again in $L^p(M\times[0,T))$ for every $p\in[1,\infty)$ due to \eqref{5.Sobu}, \eqref{5.Lpphi}, and \eqref{5.rhoX}. Thus, we also have
\begin{equation}\label{5.Lpdphi}
\nabla^3_{g_0} d\phi,\, \pt (\nabla_{g_0}\phi) \in L^p(M \times [0,T)), \qquad 1\leq p<\infty,
\end{equation}
and thus by the Sobolev embedding theorem, we can extend both $\phi$ and $d\phi$ to $M \times [0,T]$ with
\begin{equation}\label{5.Hoelderphi}
\phi,d\phi \in C^{0,\beta}(M \times [0,T]), \qquad \text{ for all } 0\leq \beta < 1.
\end{equation}
Based on the improved regularity of $u$ and $\phi$, we can now prove that also $\varrho$ and $X$, characterised by \eqref{eq:rho} and \eqref{eq:X}, respectively, have H\"{o}lder continuous extensions on $M \times [0,T]$. First, by \eqref{eq:rho}, \eqref{est:divT}, \eqref{5.Sobu}, and \eqref{5.Lpphi}, we obtain that
\begin{align*}
\Norm{-(\Deltao+2\bar{K})\varrho(t)}_{W^{-1,p}(M)} &\leq C\Norm{\deltao\big(e^{-2u}(d\phi \otimes d\phi - e(\phi,g_0)g_0)\big)}_{L^p(M)}\\
&\leq C \Norm{du(t)}_{L^p(M)} + \Norm{\nabla_{g_0} d\phi(t)}_{L^p(M)}\in L^p([0,T]).
\end{align*}
By standard elliptic theory, we thus have
\begin{equation}\label{5.Lprho1}
d\varrho \in L^p(M \times [0,T]), \qquad 1\leq p<\infty.
\end{equation}
Note that the corresponding estimate for $X$, namely
\begin{equation}\label{5.LpX1}
dX \in L^p(M \times [0,T]), \qquad 1\leq p<\infty.
\end{equation}
was already obtained in Lemma \ref{lemma:rho-X}. Differentiating \eqref{eq:rho} in time and using once again that $\pt d\mu_{g_0}=0$, we can characterise $\pt \varrho(t)$ as the unique solution of
\begin{equation*}
-(\Deltao+2\bar{K})(\pt \varrho(t)) = k(t) \quad \text{with} \quad \int_M (\pt\varrho) \dmo = 0,
\end{equation*}
where
\begin{equation*}
k(t) = \frac{d}{d\eps}\big\vert_{\eps=0}\,\Delta_{g_0(t+\eps)}\varrho(t) + \pt\big(\deltao\deltao(e^{-2u}\tf{T}(\phi,g_0))\big).
\end{equation*}
By elliptic regularity theory, we thus have that for any $p\in (1,\infty)$,
\begin{align*}
\Norm{\pt \varrho(t)}_{L^p(M)} &\leq C\Norm{k(t)}_{W^{-2,p}(M)}\\
&\leq C\Norm{\pt g_0}_{C^2(M)}\cdot \big(\Norm{\varrho(t)}_{L^p(M)} + 1\big)\\
&\quad\, + C\Norm{\pt u}_{L^p(M)} + C\Norm{\pt\nabla_{g_0}\phi}_{L^p(M)} +C,
\end{align*}
and thus that
\begin{equation}\label{5.Lprho2}
\pt\varrho \in L^p(M \times [0,T]), \qquad 1\leq p<\infty.
\end{equation}
A similar argument shows that also $X$ satisfies
\begin{equation}\label{5.LpX2}
\pt X \in L^p(M \times [0,T]), \qquad 1\leq p<\infty.
\end{equation}
From \eqref{5.Lprho1}--\eqref{5.LpX2}, we thus conclude that
\begin{equation}\label{5.HoelderrhoX}
\varrho, X \in C^{0,\beta}(M \times [0,T]), \qquad \text{ for all } 0\leq \beta < 1.
\end{equation}
We have thus proved the following result.
\begin{prop}
Let $(g(t)=e^{2u(t)}g_0(t),\phi(t))_{t\in[0,T)}$,  $T<\infty$, be a solution of \eqref{eq.flow} for which \eqref{ass:energy-density} holds and let $\varrho$ and $X$ be characterised by \eqref{eq:rho} and \eqref{eq:X}, respectively. Then $u$, $\phi$, $d\phi$, $\varrho$, and $X$ are all $\beta$-H\"{o}lder continuous with respect to the (standard) product metric $d = (d_{g_0(0)}^2 + d_{\mathbb{R}}^2)^{1/2}$, and for all $0\leq \beta <1$.
\end{prop}
\begin{remark}\label{5.rem}
This proposition implies in particular that $u$, $\phi$, $d\phi$, $\varrho$, and $X$ are also elements of the space $C^{0,\beta}_{par}(M\times [0,T])$ of functions (or vector fields or 1-forms, respectively) which are $\beta$-H\"{o}lder continuous with respect to the \emph{parabolic} distance $d_{par} = d_{g_0(0)} + d_{\mathbb{R}}^{1/2}$. Using parabolic H\"{o}lder spaces will simplify the argument in the next subsection. 
\end{remark}

\subsection{Curvature bounds and proofs of the main results}
By Remark \ref{5.rem} above and in view of the fact that $t \mapsto g_0(t)$ is Lipschitz continuous with respect to any $C^k$-norm in space, we can thus view \eqref{eq:u} as a linear parabolic equation
\begin{equation*}
\pt u - e^{-2u}\Deltao u - du(X) \in C^{0,\beta}_{par}(M \times [0,T]), \qquad \forall\, 0 \leq \beta < 1,
\end{equation*}
with H\"{o}lder continuous coefficients and right hand side. Thus, by standard parabolic theory, we can deduce that $u \in C^{2,1;\beta}_{par}(M \times [0,T])$ and in particular
\begin{equation}\label{5.C2est}
\sup_{M \times [0,T]} \Abso{\nabla_{g_0} du} < \infty.
\end{equation}
In view of the Gauss equation \eqref{Gausseq} and of \eqref{5.Sobu}, this yields
\begin{equation}\label{5.curvbound}
\sup_{M \times [0,T]} \Abs{K_g} < \infty.
\end{equation}
proving the curvature bounds claimed in Proposition \ref{prop:main-result}. From this, it is now easy to deduce the main results stated in  the introduction.

\begin{proof}[Proof of Proposition \ref{prop:main-result}]
Having already deduced the necessary curvature bounds, it remains to prove the injectivity radius bound claimed in Proposition \ref{prop:main-result}. It is well known that under a curvature bound, a lower bound on the injectivity radius is equivalent to a lower bound on the volume of (small) balls, see e.g. Proposition 1.35 and Theorem 1.36 of \cite{MT07}. In view of this fact, we obtain the desired injectivity radius bounds for the metrics $g(t)$ from the corresponding bounds for the metrics $g_0(t)$ in \eqref{5.injg0}, together with the uniform curvature bounds \eqref{5.curvbound} and the uniform bound \eqref{5.Sobu} on the conformal factor $u(t)$.\\

By standard arguments, a solution $(g,\phi)$ of \eqref{eq.flow} can always be smoothly extended in the presence of uniform bounds on the curvature, the injectivity radius and the energy density, compare with Section 6 of \cite{M12} where the corresponding result was proven in detail for the non-renormalised Harmonic Ricci Flow (in arbitrary dimension). 
\end{proof}

\begin{proof}[Proof of Theorem \ref{thm:large-alpha} and \ref{thm:blow-up}]
Recall that a solution $(\tilde{g},\tilde{\phi})$ of the (volume-preserving) Harmonic Ricci Flow \eqref{eq.RH} and its corresponding solution $(g,\phi)$ of \eqref{eq.flow} differ only by the pull-back with diffeomorphisms. Hence, Proposition \ref{prop:main-result} shows that if a solution of \eqref{eq.RH} cannot be continued smoothly past some time $T<\infty$ then the energy density cannot be bounded uniformly on $M\times [0,T)$.\\

In order to obtain Theorem \ref{thm:blow-up}, we note that a blow-up of the energy density always causes a blow-up of the curvature to happen at the same time; this result was proven by the first author in \cite[Corollary 5.3]{M12} for the non-renormalised Harmonic Ricci Flow in arbitrary dimensions and that proof applies with only minor modifications also for the renormalised flow: Indeed, a short calculation shows that 
\begin{equation*}
(\pt-\Delta_g)\big(\alpha \Abs{d\phi}_g ^2-2K_g\big)\leq C \big(\Abs{d\phi}_g ^2+1\big)\big(\Abs{K_g}+1\big)
\end{equation*}
so that a uniform bound on the curvature on $M\times [0,T)$ would imply that the energy density can grow at most exponentially in time and thus cannot blow up at $T<\infty$. This establishes Theorem \ref{thm:blow-up}.\\

In the setting of Theorem \ref{thm:large-alpha}, i.e. for $\alpha$ satisfying \eqref{eq.alphabound}, 
the Bochner-formula prevents $\abs{d\phi(x,t)}_{g(t)}^2$ from blowing up, as follows. Denoting
\begin{equation}
C_K = 2\max \big\{ K(\tau)\mid \tau\subset T_pN \textrm{ two plane, } p\in N \big\},
\end{equation}
a short computation similar to \cite[Proposition 4.4]{M12} yields
\begin{align*}
\pt \Abs{d\phi}_g^2 &= \Delta_g\Abs{d\phi}_g^2 - 2\alpha\Abs{d\phi \otimes d\phi}_g^2 - 2\Abs{\nabla_g d\phi}_g^2 + 2\Scal{\!\phantom{.}^N\Rm(\nabla_i\phi, \nabla_j\phi)\nabla_j\phi,\nabla_i\phi}\\
&\quad\, + 2(\alpha\bar{E}(t) - \bar{K})\Abs{d\phi}_g^2\\
&\leq \Delta_g\Abs{d\phi}_g^2 - (\ubar{\alpha}-C_K)\Abs{d\phi}_g^4 + 2(\bar{\alpha}\bar{E}(t)+1)\Abs{d\phi}_g^2,
\end{align*}
where we used that $\alpha(t)\in[\ubar{\alpha},\bar{\alpha}]$. We point out that the second term on the right hand side is negative, as by assumption \eqref{eq.alphabound} $C_K < \ubar{\alpha}$. Thus, using in particular also that $\bar{E}(t)$ is non-increasing, we find, for as long as the flow exists,
\begin{equation}
\Abs{d\phi}_g ^2 \leq \max\Big\{ \max_{y\in M} \Abs{d\phi(y,0)}_{g(0)}^2, \, 2\frac{\bar{\alpha}\bar{E}(0)+1}{\ubar{\alpha}-C_K} \Big\}.
\end{equation}
This finishes the proof of Theorem \ref{thm:large-alpha}.
\end{proof}

\makeatletter
\def\@listi{%
  \itemsep=0pt
  \parsep=1pt
  \topsep=1pt}
\makeatother
{\fontsize{10}{11}\selectfont

\vspace{10mm}

Reto Buzano (M\"uller)\\
{\sc School of Mathematical Sciences, Queen Mary University of London, London E1 4NS, United Kingdom}\\

Melanie Rupflin\\
{\sc Mathematical Institute, University of Oxford, Oxford OX2 6GG, United Kingdom}\\


\begin{thebibliography}{99}

\bibitem{Ab14d}
A.~Abolarinwa, \emph{Evolution and monotonicity of the first eigenvalue of p-Laplacian under the Ricci-harmonic map flow}, Preprint 2014.

\bibitem{Ab15}
A.~Abolarinwa, \emph{Entropy formulas and their applications on time dependent Riemannian metrics}, Electron. J. Math. Anal. Appl. \textbf{3}(1), 77--88, 2015.

\bibitem{Ab15b}
A.~Abolarinwa, \emph{Sobolev-type inequalities and heat kernel bounds along the geometric flow}, Afr. Mat. (to appear).

\bibitem{B13}
M.~B{\u{a}}ile{\c{s}}teanu, \emph{Gradient estimates for the heat equation under the Ricci-harmonic map flow}, Adv. Geom. (to appear).

\bibitem{B13b}
M.~B{\u{a}}ile{\c{s}}teanu, \emph{On the Heat Kernel under the Ricci Flow Coupled with the Harmonic Map Flow}, ArXiv:1309.0138.

\bibitem{BT13}
M.~B{\u{a}}ile{\c{s}}teanu and H.~Tran, \emph{Heat kernel estimates under the Ricci-harmonic map flow}, ArXiv:1310.1619.

\bibitem{CGT14}
X.~Cao, H.~Guo, and H.~Tran, \emph{Harnack Estimates for Conjugate Heat Kernel on Evolving Manifolds}, Math. Z. \textbf{281}(1), 201--214, 2015.

\bibitem{Fa13}
S.~Fang, \emph{Differential Harnack estimates for backward heat equations with potentials under an extended Ricci flow}, Adv. Geom. \textbf{13}, 741--755, 2013.

\bibitem{FZ15}
S.~Fang and T.~Zheng, \emph{The (logarithmic) Sobolev inequalities along geometric flow and applications}, J. Math. Anal. Appl. (to appear).

\bibitem{FZ15b}
S.~Fang and T.~Zheng, \emph{An upper bound of the heat kernel along the harmonic-Ricci flow}, ArXiv:1501.00639.

\bibitem{GH14}
H.~Guo and T.~He, \emph{Harnack estimates for geometric flows, applications to Ricci flow coupled with harmonic map flow}, Geom. Dedicata \textbf{169}(1), 411--418, 2014.

\bibitem{GI14}
H.~Guo and M.~Ishida, \emph{Harnack estimates for nonlinear backward heat equations in geometric flows}, J. Funct. Anal. \textbf{267}(8), 2638--2662, 2014. 

\bibitem{GI14b}
H.~Guo and M.~Ishida, \emph{Harnack estimates for nonlinear heat equations with potentials in geometric flows}, Manuscripta Math. (to appear).

\bibitem{GPT13}
H.~Guo , R.~Philipowski, and A.~Thalmaier, \emph{Entropy and lowest eigenvalue on evolving manifolds}, Pacific J. Math. \textbf{264}(1), 61--81, 2013.

\bibitem{GPT14b}
H.~Guo , R.~Philipowski, and A.~Thalmaier, \emph{On gradient solitons of the Ricci-Harmonic flow}, Acta Math. Sin. (Engl. Ser.) (to appear).

\bibitem{Li13}
Y.~Li, \emph{Eigenvalues and entropies under the harmonic-Ricci flow}, Pacific J. Math. \textbf{267}(1), 141--184, 2013.

\bibitem{List}
B.~List, \emph{Evolution of an extended Ricci flow system}, Comm. Anal. Geom. \textbf{16}(5), 1007--1048, 2008.

\bibitem{Lott}
J.~Lott, \emph{On the long-time behaviour of type III Ricci flow solutions}, Math. Ann. \textbf{339}, 627--666, 2007.

\bibitem{MT07}
J.~Morgan and G.~Tian, \emph{Ricci flow and the Poincar\'{e} conjecture}. Clay Mathematics Monographs 3, Amer. Math. Soc. (2007).

\bibitem{M09}
R.~M\"{u}ller, \emph{The Ricci flow coupled with harmonic map heat flow}, PhD thesis, ETH Z\"{u}rich, No. 18290, 2009.

\bibitem{M10}
R.~M\"{u}ller, \emph{Monotone volume formulas for geometric flows}, J. Reine Angew. Math. (Crelle) \textbf{643}, 39--57, 2010.

\bibitem{M12}
R.~M\"{u}ller, \emph{Ricci flow coupled with harmonic map flow}, Ann. Scient. \'{E}c. Norm. Sup. $4^e$ s\'{e}rie, t.~45, 101--142, 2012.

\bibitem{RT}
M.~Rupflin and P.~M.~Topping, \emph{Flowing maps to minimal surfaces}, Amer. J. Math. \textbf{138}(4), 1095--1115, 2016.

\bibitem{R-existence} 
M.~Rupflin, \emph{Flowing maps to minimal surfaces: Existence and uniqueness of solutions}, Ann. Inst. H. Poincar\'{e} Anal. Non Lin\'{e}aire \textbf{31}, 349--368, 2014.

\bibitem{RT-neg}
M.~Rupflin and P.~M.~Topping, \emph{Teichm\"uller harmonic map flow into nonpositively curved targets}, J. Diff. Geom. (to appear), ArXiv:1403.3195.
 
\bibitem{RT-horizontal}
 M.~Rupflin and P.~M.~Topping, \emph{Horizontal curves of hyperbolic metrics},  ArXiv:1605.06691. 
 
\bibitem{RT-global}
M.~Rupflin and P.~M.~Topping, \emph{Global weak solutions of the Teichm\"uller harmonic map flow into general targets}, in preparation. 
 
\bibitem{Struwe} 
M.~Struwe, \emph{Curvature flows on surfaces}. Ann. Sc. Norm. Super. Pisa Cl. Sci. (5) \textbf{1}(2), 247€--274, 2002.

\bibitem{Ta14}
H.~Tadano, \emph{A note on lower diameter bounds for closed domain manifolds of shrinking Ricci-harmonic solitons}, ArXiv:1406.2861.

\bibitem{Ta15}
H.~Tadano, \emph{A lower diameter bound for compact domain manifolds of shrinking Ricci-harmonic solitons}, Kodai Math. J. \textbf{38}(2), 302--309, 2015.

\bibitem{Ta15b}
H.~Tadano, \emph{Gap theorems for Ricci-harmonic solitons}, ArXiv:1505.03194.

\bibitem{Tromba}
A.~Tromba, \emph{Teichm\"{u}ller theory in Riemannian geometry}, Lectures in Mathematics ETH-Z\"{u}rich. Birkh\"{a}user (1992).

\bibitem{Wa12}
L.~F.~Wang, \emph{Differential Harnack inequalities under a coupled Ricci flow}, Math. Phys. Anal. Geom. \textbf{15}(4), 343--360, 2012.

\bibitem{Wi15}
M.~B.~Williams, \emph{Results on coupled Ricci and harmonic map flows}, Adv. Geom. \textbf{15}(1), 7--26, 2015.

\bibitem{Zh13}
A.~Zhu, \emph{Differential Harnack inequalities for the backward heat equation with potential under the harmonic--Ricci flow}, J. Math. Anal. Appl. \textbf{406}(2), 502--510, 2013.

\end{thebibliography}
\end{document}